\documentclass{etds}

\usepackage{hyperref}

\usepackage[T1]{fontenc}
\usepackage[utf8]{inputenc}
\usepackage[english]{babel}
\usepackage{amsmath,amssymb, tikz,enumerate,palatino}

\newcommand\A{\mathcal A}
\newcommand\Z{\mathbb Z}
\newcommand\F{\mathbb F}
\newcommand\N{\mathbb N}
\newcommand\Ne{\mathcal N}
\newcommand\M{\mathcal M}
\newcommand\R{\mathbb R}
\newcommand\U{\mathbb U}
\newcommand\B{\mathfrak B}

\newcommand{\Ms}{\mathcal M_\sigma}
\newcommand\az{{\mathcal A^{\mathbb Z}}}

\newcommand\dm{d_{\mathcal M}}

\DeclareMathOperator{\rank}{rank}
\DeclareMathOperator{\support}{supp}

\DeclareMathOperator{\Ima}{Im}

\newcommand\dual[1]{\widehat{#1}}
\newcommand\confcar[1]{\Psi\left(#1\right)}

\newenvironment{theoremintro}[1]{\begin{center}\begin{minipage}{0.8\textwidth}\textsc{Theorem} (Theorem~\ref{#1}) ---\it}{\end{minipage}\end{center}}

\theoremstyle{plain}
\newtheorem{lemma}{Lemma}
\newtheorem{theorem}{Theorem}
\newtheorem{proposition}{Proposition}
\newtheorem{corollary}{Corollary}

\theoremstyle{definition}
\newtheorem{definition}{Definition}
\newtheorem{example}{Example}

\newtheorem{remark}{Remark}

\makeatletter
\newenvironment{proof}[1][\@nil]{\proc{Proof%
     \def\tmp{#1}%
     \ifx\tmp\@nnil .
     \else ~(#1).\fi}}{\ep\medbreak}
\makeatother

\begin{document}

\ETDS{0}{0}{0}{0}
\runningheads{Randomization in Abelian Cellular Automata}{B. Hellouin de Menibus, V. Salo, G. Theyssier}
\title{Characterizing Asymptotic Randomization in\\ Abelian Cellular Automata}
\author{B. Hellouin de Menibus\affil{1}, V. Salo\affil{2}, G. Theyssier\affil{3}}
\address{\affilnum{1} Centro de Modelamiento Matem\'atico (CMM), Universidad de Chile, Chile \& Departamento de Matem\'aticas, Universidad Andrés Bello, Chile\\
  \affilnum{2}  Centro de Modelamiento Matem\'atico (CMM), Universidad de Chile, Chile\\
\affilnum{3} Institut de Mathématiques de Marseille (Université Aix Marseille, CNRS, Centrale Marseille), France
}

\recd{March 2017}

\begin{abstract}
  Abelian cellular automata (CA) are CA which are group endomorphisms of the full group shift when endowing the alphabet with an abelian group structure. A CA randomizes an initial probability measure if its iterated images weak$\mathstrut^\ast$-converge towards the uniform Bernoulli measure (the Haar measure in this setting). We are interested in structural phenomena, i.e. randomization for a wide class of initial measures (under some mixing hypotheses). First, we prove that an abelian CA randomizes in Ces\`aro mean if and only if it has no soliton, i.e. a nonzero finite configuration whose time evolution remains bounded in space. This characterization generalizes previously known sufficient conditions for abelian CA with scalar or commuting coefficients. Second, we exhibit examples of strong randomizers, i.e. abelian CA randomizing in simple convergence; this is the first proof of this behaviour to our knowledge. We show however that no CA with commuting coefficients can be strongly randomizing. Finally, we show that some abelian CA achieve partial randomization without being randomizing: the distribution of short finite words tends to the uniform distribution up to some threshold, but this convergence fails for larger words. Again this phenomenon cannot happen for abelian CA with commuting coefficients.
\end{abstract}

\section{Introduction}
\label{sec:intro}

Cellular automata, although extremely simple to define, provide a rich source of examples of dynamical systems which are not yet well understood. This is in particular true when taking a measure theoretic point of view and studying the evolution of a probability measure under iterations of a CA. The situation can be roughly depicted as follows: for non-surjective CAs, essentially all behaviors that are not prohibited by immediate computability restrictions can happen \cite{hellouindemenibus2016,BoyerDPST15,DelacourtM15}; for the surjective case, various forms of rigidity are observed (see \cite{Pivato2009} for an overview). In particular, since the pioneering work of Lind and Miyamoto on the `addition modulo 2' CA \cite{miyamoto79,Lind84}, many CAs of algebraic origin were shown to behave like randomizers \cite{Maassetal,PivatoYassawi1,PivatoYassawi2,mmpy2006,host2003uniform}, \textit{i.e.} they converge in Ces\`aro mean or in density to the uniform Bernoulli measure from any initial probability measure from a large class $\mathcal{C}$. In \cite{miyamoto79,Lind84}, the class $\mathcal{C}$ is Bernoulli measures of full support. It was later extended to full support Markov measures or N-step Markov processes, measures with complete connections and summable decay of correlations and harmonically mixing measures \cite{Pivato2009} and more \cite{pivato2006, sobottka}.

Apart from specific examples (like in \cite{Maass1999}), the class of CA where randomizing behavior has been shown is essentially contained in that of `linear' CA defined on an abelian group alphabet by
\[F(x)_i = \sum_{j\in V} \theta_j(x_{i+j}),\]
where $\theta_j$ are \emph{commuting} endomorphisms (most of the time automorphisms or scalar coefficients). Furthermore, the type of convergence considered has always been Ces\`aro mean or convergence in density.\\

In this paper we consider the class of harmonically mixing measures and the class of abelian CAs which are like `linear' CA described above but \emph{without} the assumption of commutation of endomorphisms. Our first main result is a complete characterization of randomization in density in that setting.

\begin{theoremintro}{thm:main}
  An abelian CA $F$ randomizes in density any harmonically mixing measure if and only if it does not possess a soliton, \textit{i.e.} a nonzero finite configuration whose set of nonzero cells stays within a bounded diameter under iterations of the CA.
\end{theoremintro}

We show that this theorem extends the most general previous result \cite{PivatoYassawi2} and allows us to easily prove randomization for particular examples, even in the setting of non-commutative coefficients \cite{Maass1999}, hence answering a question of \cite{Pivato2009}.

Our approach uses tools from harmonic analysis using a similar approach to the work of Pivato and Yassawi on diffusion of characters \cite{PivatoYassawi1,PivatoYassawi2}. We rely on the abelian structure of the considered CA to reduce randomization to a combinatorial property of diffusivity. More precisely, we define a dual CA on the (Pontryagin) dual group and show that diffusion in the dual is equivalent to randomization and that the diffusion property is preserved by duality. Finally, we prove the equivalence between diffusivity and the absence of solitons, not by using the abelian structure but by general combinatorial properties of surjective CA. This allows us to go beyond the commuting coefficient case, which was treated in previous works by a careful analysis of binomial coefficients of the iterates of $F$. 

We also prove the existence of a stronger form of randomization where taking subsequences of density $1$ or Ces\`aro mean is not necessary:

\begin{theoremintro}{ref:strongrandomize}
  There exist abelian CAs that randomize in simple convergence any harmonically mixing measure.
\end{theoremintro}

This answers a question of \cite{Kari2015} (Question 59). Experiments on small surjective CAs \cite{These, Taati} suggest that this strong form of randomization is the most common, that it occurs as well on nonabelian CA, and even that randomization occurring only in density (or Ces\`aro mean) might be an artifact of abelian CAs. This confirms the importance of the non-commutative coefficients case, since we also prove that abelian CA with commuting coefficients cannot achieve such a strong form of randomization.\medskip

\begin{center}
\begin{figure}[!h]
  \begin{tabular}{ccc}
    \rotatebox[origin=c]{180}{\includegraphics[width=.3\textwidth]{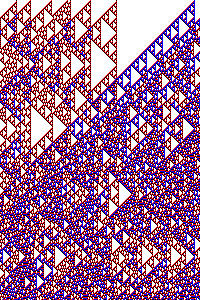}}&\rotatebox[origin=c]{180}{\includegraphics[width=.3\textwidth]{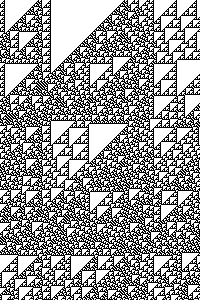}}&\rotatebox[origin=c]{180}{\includegraphics[width=.3\textwidth]{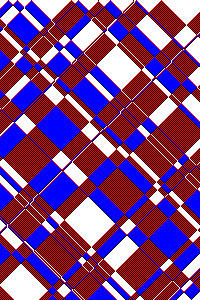}}\\
    Strong randomizer & Randomizer in density & ``Cellwise'' randomizer
  \end{tabular}
  \caption{Three forms of randomization: $F_2$ (defined in Section~\ref{sec:strongrandom}), addition modulo 2, and $I_{\Z_2}$ (defined in Section~\ref{sec:fixedlength}). The direction of time is upward. The CA are iterated on an initial configuration drawn according to a Bernoulli measure with 95\% white (state $0$). The addition modulo $2$ does not converge directly because the image measure is far from the uniform measure around times $t = 2^n$ (see Theorem~\ref{thm:commutingcase}). $I_{\Z_2}$ randomizes individual cells but not cylinders of length $2$ (see Proposition~\ref{prop:cellwise}).}
\end{figure}
\end{center}

The results above are stated as randomization for the class of harmonically mixing measures. We do not investigate when there are more randomized measures (as in~\cite{pivato2006}), however we show that there can not be fewer: if an abelian CA randomizes in density full-support Bernoulli measures, then it randomizes in density all harmonically mixing measures. Interestingly, the rigidity is even stronger for abelian CA with commutative coefficients: we prove that if the frequency of individual states is randomized (in density), then the CA is fully randomizing in density. In the case of non-commutative coefficients, we can have partial randomization: we give examples for any $K$ of abelian CA which do randomize all cylinders up to size $K$ but fail to randomize completely. This suggests that experimental work on randomization in general CA should be done with care: randomization might fail in non-obvious ways on long-range correlations.\\

The paper is organized as follows: in Section~\ref{sec:defs} we recall basic definitions and tools about measure theoretic aspects of cellular automata; in Section~\ref{sec:chidiff} we study the evolution of ranks of characters under iterations of abelian CAs, the property of character diffusion and its link with randomization; in Section~\ref{sec:dualitysolitons} we define the dual of an abelian CA, show that duality preserves diffusivity in density and link this property with the absence of solitons; in Section~\ref{sec:mainthm} we establish our main result, a characterization of randomization in density through the absence of solitons; in Section~\ref{sec:other}, we exhibit a class of examples of strong randomization and randomization up to fixed-length cylinder, and we show that these behaviors are specific to CA with non-commuting coefficients; finally, in Section~\ref{sec:open} we give some directions for further research on this topic.

\section{Definitions and tools}
\label{sec:defs}

Throughout this paper we will state our results for dimension $1$, but they extend straightforwardly to the $d$-dimensional case. Our convention on natural number is $0 \in \N$.

Let $\A$ be a finite alphabet. We define $\A^\ast = \bigcup_{n \in \N} \A^n$ to be the set of finite \emph{words}, and $\az$ the set of (one-dimensional) \emph{configurations}. For a finite set $U\subset \Z$ and $u\in\A^U$, define the \emph{cylinder}:
\[[u]_U = \{x\in\az\ :\ x|_U = u\}.\]
For $u\in\A^n$ and $k\in\Z$, also define $[u]_k = [u]_{\{k,\dots, k+n-1\}}$ and $[u] = [u]_0$.

We endow $\A^\Z$ with the product topology, which is metrizable using the \emph{Cantor distance}:
\[\forall x,y\in\az, d(x,y) = 2^{-\Delta(x,y)} \quad \text{where}\quad \Delta(x,y) = \min\{|i| : x_i \neq y_i\}.\]

The \emph{shift map} is defined by 
\[\forall x\in\az, \sigma(x) = (x_{i+1})_{i\in\Z}.\]

A \emph{cellular automaton} is a pair $(\A, F)$ where $F : \az\to \az$ is a continuous function that commutes with the shift map (i.e. $F\circ \sigma = \sigma\circ F$). Equivalently, $F$ is defined by a finite neighborhood $\Ne\subset \Z$ and a \emph{local rule} $f : \A^\Ne \to \A$ in the sense that:
\[\forall x\in\az,\forall i\in\Z, F(x)_i = f(x_{i+\Ne}).\]

\newcommand\confplus{+}

Let ${(\A,+)}$ be an abelian group and $0$ its neutral element. 
%If $0$ denotes the neutral element of the group ${(\A,+)}$
A \emph{finite configuration} is a configuration ${x\in\az}$ such that ${x(i)=0}$ for all ${i\in\Z}$ except on a finite set.
If $x$ is a finite configuration, we define its \emph{support} by ${\support(x) = \{i\in\Z : x(i)\neq 0\}}$ and its \emph{rank} by ${\rank(x)=|\support(x)|}$. Note that the set of finite configurations is dense in ${\az}$. The notion of finite points makes sense also when $\A$ has no group structure, assuming a choice of zero element $0 \in \A$ has been made.

An \emph{abelian cellular automaton} $F$ is a cellular automaton which is an endomorphism for $(\az, \confplus)$ (componentwise addition):

\[\forall x,y\in\az, F(x\confplus y) = F(x)\confplus F(y)\]
\newcommand\alfaendo{\phi}
\newcommand\confendo{\overline{\alfaendo}}

Equivalently, $F$ is a finite sum of shifts composed with endomorphisms of ${(\A,+)}$. More precisely, there is a finite ${\Ne\subset\Z}$ and a collection ${(\alfaendo_i)_{i\in\Ne}}$ of endomorphisms of ${(\A,+)}$ such that:

\[F = \sum_{i\in\Ne} \confendo_i\circ\sigma^{i}, \qquad\mbox{where}\qquad \confendo_i: \begin{array}{ccc} \az &\to &\az\\x&\mapsto&\bigl(\alfaendo_i(x(j))\bigr)_j\end{array}.\] Note that the image of a finite configuration is always a finite configuration. In particular, $0 \in \A$ is a \emph{quiescent state}, meaning $F(0^\Z) = 0^\Z$ where $0^\Z$ denotes the constant-$0$ configuration.

  We also define addition on abelian CA by $F+F' : x\mapsto F(x)+F'(x)$.

We say that $F$ has \emph{commuting endomorphisms} if the endomorphisms ${\alfaendo_i}$ commute pairwise. Supposing $\A$ is a vector space over a finite field ${\F_p}$ turns out to be a source of simple yet illustrative examples. 
As said above, we are particularly interested in the non-commuting case. We will illustrate our results with the two representatives $F_2$ and $H_2$ defined over ${\A=\F_2^2}$ by %\TODO{voir}

%\TODO{better names?} 

\begin{align*}
F_2(x)_i &=
\begin{pmatrix}
        0 & 1\\
        1 & 0
      \end{pmatrix}\cdot x_i + 
      \begin{pmatrix}
        1 & 0\\
        0 & 0
      \end{pmatrix}\cdot x_{i+1}\\
H_2(x)_i &= 
      \begin{pmatrix}
        1 & 0\\
        0 & 0
      \end{pmatrix}\cdot x_{i-1}+ \begin{pmatrix}
        0 & 1\\
        1 & 0
      \end{pmatrix}\cdot x_i + 
      \begin{pmatrix}
        1 & 0\\
        0 & 0
      \end{pmatrix}\cdot x_{i+1}
\end{align*}
where elements of $\A$ are seen as vectors and matrix notation is used to denote endomorphisms of $\A$.

\subsection{Cellular automata acting on probability measures}

Let $\M(\az)$ be the space of probability measures on the Borel sigma-algebra of $\az$. In particular, we consider $\Ms(\az)$, the subset of all $\sigma$-invariant measures. Here are a few examples that we mention throughout the paper:

\begin{description}
\item[Bernoulli measure] Take $(\beta_i)_{i\in\A} \in [0,1]^\A$ such that $\sum_i \beta_i=1$. Let $\beta$ be the usual Bernoulli measure of parameters $(\beta_i)$ on $\A$. The \emph{Bernoulli measure} of parameters $(\beta_i)$ on $\az$ is defined as $\mu = \otimes_\Z\beta$; that is, each cell is drawn in a i.i.d manner and distributed as a Bernoulli measure. In other words,
\[\forall u\in\A^\ast, \mu([u]) =\prod_{0\leq i<|u|}\beta_{u_i}.\]
A particularly important example is the uniform Bernoulli measure on $\az$, denoted $\lambda$, which is the Bernoulli measure of parameters $\left(\frac 1{|\A|}\right)_{i\in\A}$.
\item[Markov measure] Let $(p_{i,j})_{i,j\in\A}$ be a nonnegative matrix satisfying $\sum_j p_{ij} = 1$ for all $i$, and
$(\mu_i)_{i\in\A}$ an eigenvector associated with the eigenvalue 1 (the choice being unique if the matrix is irreducible). 
The associated \emph{two-step Markov measure} is defined as 
\[\forall u\in\A^\ast, \mu([u]) = \mu_{u_0}\prod_{0\leq i<|u|}p_{u_iu_{i+1}}.\] This can be extended to $n$-step Markov measures.
\end{description}

The weak-* topology on $\M(\az)$ is metrisable. A possible metric is given by the distance:
\[\dm(\mu, \nu) = \sum_{k\in\N} \frac 1{2^k}\max_{u\in\A^{2k+1}} |\mu([u]_{-k}) - \nu([u]_{-k})|.\]

A cellular automaton $(\A, F)$ yields a continuous action on the space of probability measures $\M(\az)$:
\[\text{For any Borel set } U,\ F\mu(U) = \mu(F^{-1} U)\]

Given a subset ${X\subseteq\N}$ of natural numbers, its \emph{lower density} is defined as ${\liminf_n \frac{|\{i\in X, i\leq n\}|}{n}}$. 
We prove a diagonalization lemma for sequences of lower density one, stating, informally, that the intersection of countably many sequences of lower density one ``eventually has lower density one''.

\begin{lemma}
\label{lem:DensityDiagonalization}
Suppose that $I$ is countable, and for all $n \in I$, $N_n \subset \N$ is a set with lower density one. Then there there exists a set $N \subset \N$ of lower density one such that for all $n \in I$, $N \cap [k, \infty) \subset N_n \cap [k, \infty)$ for all large enough $k$.
\end{lemma}

\begin{proof}
We can assume $I = \N$. The intersection of finitely many sets of lower density one has lower density one. Thus, we may assume the $N_n$ form a decreasing sequence (with respect to inclusion) by replacing each $N_n$ by $N_0 \cap N_1 \cap \cdots \cap N_n$.

Now, let $m_0 = 0$ and pick an increasing sequence of natural numbers $(m_n)_{n \geq 1}$ such that $|N_n \cap [0, m)| \geq m (1 - 1/n)$ for all $m \geq m_n$, using the fact $N_n$ has lower density one. Define $N \cap [m_{n-1}, m_n) = N_n \cap [m_{n-1}, m_n)$ for all $n$. Then
\begin{align*}
 |N \cap [0, m)| &= |N_1 \cap [0, m_1)| + \cdots + |N_\ell \cap [m_{\ell-1}, m_\ell)| + |N_{\ell+1} \cap [m_{\ell}, m)|\\
& \geq |N_\ell \cap [0, m)| \\
&\geq m (1 - 1/\ell) 
\end{align*}
where $\ell$ is maximal such that $m_{\ell} < m$, and where the first inequality follows because the $N_n$ form a decreasing sequence under inclusion.
\end{proof}

Considering the iterated action of a cellular automaton on an initial measure $\mu$, we distinguish various forms of convergence:

\begin{itemize}
\item $(F^t\mu)_{t\in\N}$ \emph{converges} to $\nu$ if $F^t\mu\to \nu$ (for the weak-$\ast$ convergence); equivalently, $F^t\mu([u])\to\nu([u])$ for every finite word $u$;
%  \TODO{strongly converges ?}
\item $(F^t\mu)_{t\in\N}$ \emph{converges in Cesàro mean} to $\nu$ if $\displaystyle\frac 1T\sum_{t=0}^{T-1}F^t\mu\to \nu$; equivalently, if $\displaystyle\frac 1T\sum_{t=0}^{T-1}F^t\mu([u])\to \nu([u])$  for every finite word $u$;
\item $(F^t\mu)_{t\in\N}$ \emph{converges in density} to $\nu$ if there exists an increasing sequence $(\varphi(t))_{t\in\N}$ of lower density $1$ such that $F^{\varphi(t)}\mu\to \nu$; equivalently by Lemma~\ref{lem:DensityDiagonalization}, if for every finite word $u$ there exists an increasing sequence $(\varphi_u(t))_{t\in\N}$ of lower density $1$ such that $F^{\varphi_u(t)}\mu([u])\to \nu([u])$;
\item $(F^t\mu)_{t\in\N}$ \emph{converges on cylinders of support $\subset \U$} to $\nu$ if $\mu(\cdot\ |\ \B_\U) \to \nu (\cdot\ |\ \B_\U)$ where $\B_\U$ is the Borel $\sigma$-algebra generated by the cylinders of support $\subset \U$ (this can be seen as convergence of measures of $\M(\A^\U)$); equivalently, $F^{t}\mu([u])\to \nu([u])$ for every word $u$ with $\support(u)\subset \U$.
\end{itemize}

Recall that, in a context where the alphabet is $\A$, $\lambda$ is the uniform Bernoulli measure on $\az$.

\begin{definition}[Randomization]
Let $F : \az\to\az$ be a cellular automaton and $\M\subset \M(\az)$ be a class of initial measures.\bigskip

$F$ \emph{strongly randomizes} $\M$ (resp. in Cesàro mean, in density, on cylinders of support $\U$) if, for all $\mu\in\M$, $(F^t\mu)_{t\in\N}$ converges to $\lambda$ (resp. in Cesàro mean, in density, on cylinders of support $\U$).
\end{definition}

% Strong randomization is sometimes also called \emph{direct randomization} or \emph{randomization in simple convergence} to avoid ambiguity, especially since we are going
%randomization in Cesàro mean is sometimes called simply randomization by some authors (\cite{pivatoyassawi1}). 

\begin{proposition}
Strong randomization implies all other forms of randomization, and randomization in Cesàro mean is equivalent to randomization in density for $\sigma$-invariant measures.
\end{proposition}

\begin{proof}
  The first point is clear. The second point stems from the fact that the uniform Bernoulli measure $\lambda$ is an extremal point of $\Ms(\az)$ and $\Ms(\az)$ is compact. We prove this point by contraposition.
  Assume that $(F^t\mu)_{t\in\N}$ does not converge to $\lambda$ in density. Then there must exist some $\varepsilon>0$ and some sequence $(\varphi(t))$ of upper density $\alpha>0$ such that $F^{\varphi(t)}\mu \notin B(\lambda, \varepsilon)$, where $B(\lambda, \varepsilon)$ is the open ball of radius $\varepsilon$ centered on $\lambda$ (otherwise for any $n>0$ the set of times $t$ with ${F^{t}\mu \in B(\lambda, \frac{1}{n})}$ would be of density $1$ and by Lemma~\ref{lem:DensityDiagonalization}, $(F^t\mu)_{t\in\N}$ would converge to $\lambda$ in density). Therefore there exists a sequence of times $(T_i)_{i\in\N}$ such that $\frac {T_i}{\varphi(T_i)} \to \alpha$. Then: 

  \[\frac 1{\varphi(T_i)+1}\sum_{t=0}^{\varphi(T_i)}F^t\mu = \frac 1{\varphi(T_i)+1}\sum_{t=0}^{T_i}F^{\varphi(t)}\mu + \frac 1{\varphi(T_i)+1}\sum_{\substack{t=0\\t\notin\varphi(\N)}}^{\varphi(T_i)}F^t\mu .\]

  Let $\mathcal C$ be the convex hull of $\Ms(\az)\backslash B(\lambda, \varepsilon)$. By compactness, this sequence admits accumulation points which must be of the form $\alpha\nu + (1-\alpha)\eta$ for some $\nu\in\mathcal C$ and $\eta\in\Ms(\az)$. However, since $\lambda$ is extremal in $\Ms(\az)$, $\lambda \notin \mathcal C$, so that $\lambda \neq \nu$, and $\lambda \neq \alpha\nu + (1-\alpha)\eta$. In other words, the sequence $(\frac 1{T+1} \sum F^t\mu)_{t\in\N}$ admits some accumulation point which is not $\lambda$, and we conclude.
\end{proof}
\subsection{Fourier theory}

\begin{definition}[Character]
A \emph{character} of a topological group $\mathcal G$ is a continuous group homomorphism $\mathcal G \to \mathbb T^1$, where $\mathbb T^1$ is the unit circle group (under multiplication). Denote by $\dual{\mathcal G}$ the group of characters of $\mathcal G$ under elementwise multiplication.
\end{definition}

The following result is well-known (see e.g. \cite{Deitmar}, Lemma 4.1.3.):

\begin{proposition}\label{prop:isodual}
  Any finite abelian group $\mathcal G$ is isomorphic to its dual $\dual{G}$.
\end{proposition}

If $\A$ is a finite abelian group, $\widehat{\az}$ is in bijective correspondance with the sequences of $(\dual{\A})^\Z$ whose elements are all $\mathbf 1$ except for a finite number. That is, $\chi \in \widehat{\az}$ can be written as $\chi(x) = \prod_{k\in\Z} \chi_k(x_k)$ where all but finitely many elements are equal to 1. In this context we call the elements of $\widehat\A$ \emph{elementary characters}.

\begin{definition}
Let $\chi \in \widehat{\az}$ and $(\chi_i)_{i\in\Z}$ its decomposition in elementary characters. 
The \emph{support} of $\chi$ is $\support(\chi) = \{i\in\Z\ :\ \chi_i \neq 1\}$. Its \emph{rank} is $\rank(\chi) = |\support(\chi)|$.
\end{definition}

\begin{definition}[Fourier coefficients, or Fourier-Stieltjes transform]
The \emph{Fourier coefficients} of a measure $\mu\in\M(\az)$ are given by:
\[\dual\mu[\chi] = \int_\az \chi d\mu\]
for all characters $\chi\in\widehat{\az}$. For a character $\chi = \prod_{k\in S}\chi_k$ (where $S = \support(\chi)$), this can be rewritten as a finite sum:
\[\dual\mu[\chi] = \sum_{u\in\A^S} \prod_{k\in S}\chi_k(u_k)\cdot\mu([u]_S)\]
\end{definition}

The Fourier coefficients of $\mu$ completely characterize it. They also behave well with regard to convergence in (weak-*) topology:

\begin{theorem}[L\'evy's continuity theorem]
Let $G$ be a locally compact abelian group and $\mu_1, \mu_2, \dots \mu_\infty \in \M(G)$. Then:
\[\mu_n \to \mu_\infty \mbox{ in the weak-* topology} \quad \Longleftrightarrow \quad\forall \chi\in\widehat{G},\ \dual{\mu_n}[\chi] \to \dual{\mu_\infty}[\chi].\]
\end{theorem}

This theorem was first introduced in \cite{Levy} (in French). It has been extended to locally compact abelian groups in \cite{Varadhan}, and to more general settings which are out of the scope of this article.

\begin{definition}[Harmonically mixing measure]
$\mu \in \M(\az)$ is harmonically mixing if, for all $\varepsilon>0$, there exists $R>0$ such that $\rank(\chi)>R \implies \dual{\mu}[\chi]<\varepsilon$.
\end{definition}

Throughout this paper, we sometimes omit to specify the class of initial measures, which is always the class of harmonically mixing measures.

\begin{proposition}
%Let $\A = (\Z/2\Z)^2$.
Let $\A$ be any finite abelian group. Any Bernoulli or ($n$-step) Markov measure on $\az$ with nonzero parameters is harmonically mixing.
\end{proposition}

This is %a particular case of
Propositions 6 and 8 and Corollary 10 in \cite{PivatoYassawi1}.

\section{Character-diffusivity}
\label{sec:chidiff}

Let $\chi \in \dual{\az}$ and let $F : \az \to \az$ be an abelian cellular automaton. Then $\chi \circ F$ is a character (composition of continuous group homomorphisms). One of the central ideas introduced in \cite{PivatoYassawi1} is to focus on the evolution of the rank of characters under the action of $F$ in order to establish randomization in density of harmonically mixing measures. They introduce the following notion of diffusivity over characters. We call it character-diffusivity to clearly distinguish it from the notion of diffusivity we will introduce later.

\newcommand\chidiffusive{character-diffusive }

\begin{definition}[Character diffusion]
  Let $F : \az\to\az$ be an abelian cellular automaton. We say $F$ \emph{strongly diffuses} a character $\chi \in \widehat{\az}$ if ${\rank(\chi\circ F^t) \to \infty}$, and \emph{diffuses $\chi$ in density} if if the convergence occurs along an increasing sequence of times of lower density one. We say $F$ is \emph{strongly \chidiffusive} if it strongly diffuses every nontrivial character, and define \emph{character-diffusivity in density} analogously.
\end{definition}

\begin{definition}
A measure $\mu$ is \emph{strongly nonuniform} if $\mu[\chi]\neq 0$ for all characters $\chi\in\widehat{\az}$.
\end{definition}

\begin{example}
  \label{ex:nonuniform}
A Bernoulli measure $\mu = \otimes \beta$ whose parameters are all equal except one is strongly nonuniform. Assume that $\beta(a) = c$ for all $a\in\A$ except for $\beta(a')\neq c$. Let $\chi_k$ be a nontrivial elementary character, i.e. a character of $\A$. 
\[\mu[\chi_k] = \sum_{a\in\A}\beta(a)\chi_k(a) = (\beta(a')-c)\chi_k(0) + c\sum_{a\in\A}\chi_k(a) = \beta(a')-c \neq 0\] by hypothesis.
It follows that for every character $\chi = \prod_k \chi_k$, we have $\mu[\chi] = \prod_k\mu[\chi_k] \neq 0$ where we are using the fact that $\mu$ is a Bernoulli measure.

An example of a non-strongly nonuniform measure is any measure of the form $\lambda \times \mu$ on $(\A\times \mathcal B)^\Z$, where $\lambda$ is the uniform Bernoulli measure on $\az$.
\end{example}

The following proposition completes Theorem~12 of \cite{PivatoYassawi1} by giving an equivalence between character-diffusivity and randomization. It also shows that randomization is a structural phenomenon, in the sense that it cannot happen on individual initial measures without happening on a large class.

\begin{definition}
A Bernoulli measure $\otimes \beta$ is \emph{nondegenerate} if the support of $\beta$ has nontrivial intersection with at least two cosets of every proper subgroup of $\A$.
\end{definition}

\begin{proposition}
\label{prop:chidiffusive}
Let $F$ be an abelian cellular automaton. The following are equivalent:
\begin{enumerate}[(i)]
\item $F$ is \chidiffusive;
\item $F$ randomizes the class of harmonically mixing measures;
\item $F$ randomizes the class of nondegenerate Bernoulli measures;
\item $F$ randomizes some strongly nonuniform Bernoulli measure.
\end{enumerate}
This equivalence holds for all three kinds of character-diffusivity and randomization, that is: strong character-diffusivity/randomization, character-diffusivity/randomization in density, and character-diffusivity for characters of support $\subset\U$/randomization on cylinders of support $\subset\U$ for any $\U\subset \Z$.
\end{proposition}

\begin{proof}\paragraph{$(i)\Rightarrow (ii)$: }Assume $F$ is strongly \chidiffusive. Let $\mu$ be a harmonically mixing measure and $\chi$ any nontrivial character of $\az$. Since $F$ is strongly \chidiffusive, $\rank(\chi\circ F^t) \underset {t\to\infty}\longrightarrow \infty$. Since $\mu$ is harmonically mixing, it follows that $F^t\mu[\chi] = \mu[\chi\circ F^t] \underset {t\to\infty}\longrightarrow 0 = \lambda[\chi]$. Since this is true for any character $\chi$, we have by L\'evy's continuity theorem that $F^t\mu \underset {t\to\infty}\longrightarrow \lambda$.

For randomization in density, do the same proof, where each convergence is taken along a subsequence of upper density $1$. 

For randomization for characters of support $\subset \U$, the same argument shows that $F^t\mu[\chi] = \mu[\chi\circ F^t] \underset {t\to\infty}\longrightarrow 0 = \lambda[\chi]$ for any nontrivial character $\chi$ with support in $\U$. There is a bijection between the characters of support $\subset \U$ and $\widehat{\A^{\U}}$, and a conditional measure $\mu(\cdot\ |\ \B_\U)$ can be seen as a measure of $\M(\A^\U)$. Applying L\'evy's continuity theorem to $\A^{\U}$, it follows that $F^t\mu(\cdot\ |\ \B_\U) \to \lambda(\cdot\ |\ \B_\U)$.

\paragraph{$(ii)\Rightarrow (iii)$: } We prove that any nondegenerate Bernoulli measure $\mu = \otimes \beta$ is harmonically mixing. First note that for any elementary character $\chi_0 \neq \textbf{1}$, we have:
\begin{align*}
\mu[\chi_0] = \int_{\az}\chi d\mu = \sum_{a\in\A}\chi_0(a)\beta(a).
\end{align*}

We claim that there exist $a, b \in \A$ such that $\beta(a), \beta(b) > 0$ and $\chi_0(a) \neq \chi_0(b)$. To see this, let $K \leq \A$ be the kernel of $\chi_0$. Since $\chi_0 \neq \textbf{1}$, $\chi_0(g) \neq 1$ for some $g \in \A$, and thus $K < \A$. Then by the assumption that $\beta$ is nondegenerate, there exist $a, b$ in the support of $\beta$ such that $aK \neq bK$. Thus $\beta(a), \beta(b) > 0$ and $\chi_0(a) \neq \chi_0(b)$.

Now the existence of $a, b$ shows that $\sum_{a\in\A}\chi_0(a)\beta(a)$ is a non-trivial convex combination of points on the unit circle, so by the strict convexity of the unit circle, $\mu[\chi_0]$ is a non-extremal point of the unit disk, that is to say $|\mu[\chi_0]|<1$.

Define $m = \max\{|\mu[\chi_0]|\ :\ \chi_0\in\widehat \A\backslash \textbf 1\} < 1$. For any character $\chi = \prod_{i\in\Z}\chi_i$,
\[|\mu[\chi]| = \left|\prod_{i\in\Z}\mu[\chi_i]\right| \leq m^{\rank(\chi)},  \]
where the first equality comes from the fact that $\mu$ is a Bernoulli measure. This implies that $\mu$ is harmonically mixing.

\paragraph{$(iii)\Rightarrow (iv)$: } See e.g. the first measure in Example~\ref{ex:nonuniform}.

\paragraph{$(iv)\Rightarrow (i)$: } Assume $F$ is not strongly \chidiffusive and take a character $\chi\neq \textbf{1}$ such that $\rank(F^t\circ \chi) \nrightarrow \infty$. This means that there exists $C\in\N$ and a subsequence $\varphi$ such that $\rank(\chi\circ F^{\varphi(t)}) \leq C$. Since there is a finite number of elementary characters, we have for any strongly nonuniform Bernoulli measure $\mu$: \[|F^{\varphi(t)}\mu[\chi]| \geq m^C \quad \text{where} \quad m = \min\{|\mu[\chi_0]| : \chi_0\in\dual{A}\} >0.\] Therefore $F^t\mu[\chi]\nrightarrow 0$, which implies that $F^t\mu \nrightarrow \lambda$. We conclude by contraposition.

For randomization in density, do the same proof along a sequence of times with positive upper density.

For randomization of cylinders of support $\subset \U$, we can do the same proof for any character of support $\subset \U$.
\end{proof}

\begin{remark}
  The strongly nonuniform hypothesis is necessary to prevent the following kind of counterexample: take $\A = (\Z/2\Z)^2$, $F' = F\times Id$ where $F$ is a strongly randomizing cellular automaton on $(\Z/2\Z)^\Z$ (such as $F_2$, as we prove later) and $\mu = \nu\times\lambda$ where $\nu$ is any harmonically mixing measure and $\lambda$ is the uniform Bernoulli measure on $(\Z/2\Z)^\Z$. Then $F'$ strongly randomizes $\mu$, but does not strongly randomize any measure whose second component is nonuniform. 
\end{remark}

\begin{definition}[Dependency function]
  To any abelian CA $F : \az\to\az$ we associate a \emph{dependency function}:
  \[\forall (t,i)\in\N\times\Z,\ \Delta_F(t,i) =
  \left\{ \begin{array}{ccc}\A&\to&\A\\ q&\mapsto& F^t(x_q)_i\end{array}\right.\]
where $x_q$ is the configuration worth $q$ at position $0$ and $0$ everywhere else.
\end{definition}

Notice that by linearity, we have:
\begin{equation}\forall x\in\az,\ \forall t\in\N,\ F^t(x)_z = \sum_{j\in\Z} \Delta_F(t,z-j)(x_j),\label{eq:dependency}\end{equation}
where only a finite number of terms are nonzero.

In the following lemma, we prove that the support of the image of a fixed character at time $t$ is entirely determined by the local dependency diagram.

\begin{lemma}
  \label{lem:localrank}
  Let $F$ be an abelian cellular automaton, and $\chi$ be a character whose support is included in $[0,m]$ for ${m\geq 0}$. If there are $(t_1,z_1)$ and $(t_2,z_2)$ such that:
  \[\forall z\in[0,m],\ \Delta_F(t_1,z_1+z)=\Delta_F(t_2,z_2+z)\]
  then 
  \[-z_1\in\support(\chi\circ F^{t_1}) \Leftrightarrow -z_2\in\support(\chi\circ F^{t_2})\]
\end{lemma}
\begin{proof}
By Equation~\ref{eq:dependency},
\begin{align*}\chi\circ F^t(x) &= \prod_{i\in\Z}\chi_i\left(\sum_{j\in\Z}\Delta_F (t,i-j) (x_j)\right)\\
&= \prod_{i\in\Z}\prod_{j\in\Z}\chi_i\circ\Delta_F (t,i-j) (x_j)\\
&= \prod_{j\in\Z}\left(\prod_{k\in\Z}\chi_{k+j}\circ\Delta_F (t,k)\right) (x_j)
\end{align*}
where the last step is obtained by rewriting the sum: $k=i-j$. It follows: 
\begin{align*}-z_1\in\support(\chi\circ F^{t_1}) &\Leftrightarrow \prod_{k\in\Z}\chi_{k-z_1}\circ\Delta_F (t_1,k) \neq 0\\
%&\Leftrightarrow \prod_{k = 0}^m \chi_{k}\circ\Delta_F (t_1,z_1+k) \neq 0\\
%&\Leftrightarrow \prod_{k = 0}^m \chi_{k}\circ\Delta_F (t_2,k+z_2) \neq 0\\
&\Leftrightarrow \prod_{k \in \Z} \chi_{k-z_2}\circ\Delta_F (t_2,k) \neq 0\\
&\Leftrightarrow -z_2\in\support(\chi\circ F^{t_2})
\end{align*}
where the second step uses the hypothesis of the lemma and the fact that $\chi_{k-z_1}=0$ whenever $k-z_1\notin[0,m]$.
\end{proof}

Given an abelian CA $F$ and ${t\in\N}$, denote by $d(t)$ the number of non-trivial dependencies of $F$ at time $t$:
\[d(t) = \bigl|\{z\in\Z : \Delta_F(t,z) \neq 0\}\bigr| \quad \text{(the zero map)}\]

Following \cite{PivatoYassawi1,PivatoYassawi2}, we introduce \emph{isolated bijective dependencies} which provide useful lower-bounds on the rank of image of characters under the action of $F$.

\begin{definition}[Isolated dependency]
For ${k\geq 1}$, a $k$-isolated dependency is a pair ${(t,z)\in\N\times\Z}$ such that:
\begin{enumerate}
\item ${\Delta_F(t,z)}$ is a bijective dependency;
\item ${\Delta_F(t,z+i)} = 0$ for ${1 \leq i\leq k}$.
\end{enumerate}
We denote by ${S_k(t)}$ the set of $k$-isolated dependencies at time $t$ (\textit{i.e.} of the form ${(t,z)}$) and ${s_k(t)=|S_k(t)|}$.
\end{definition}
This concept of $k$-isolated dependencies is also used in \cite{pivato2006} to define dispersion mixing measure and dispersive CA.
%\TODO{cite Pivato's work on refinements of this notion for particular CAs.}
The main techniques of \cite{PivatoYassawi1} (proof of Therorem 15) and \cite{PivatoYassawi2} ($V$-separating sets) is essentially to use $s_k$ as a lower bound for the rank of characters under the iteration of an abelian CA. 
\begin{proposition}
  \label{prop:rankdeps}
  Let $F$ be an abelian cellular automaton, $\chi\neq \textbf{1}$ a
  character, and $k$ the diameter of $\support(\chi)$. Then we have:
  \[\forall t\in\N,\ s_{k-1}(t)\leq \rank(\chi\circ F^t)\leq \rank(\chi)\cdot d(t)\]
\end{proposition}
\begin{proof}
The rank being invariant by translation, we can suppose that $\support(\chi)\subset {[0,k-1]}$ and that $\chi_0\neq \textbf{1}$. By Equation~\ref{eq:dependency}, we have:
\[\chi\circ F^t(x) = \prod_{i=0}^{k-1}\prod_{j\in\Z}\chi_i\circ\Delta_F (t,i-j) (x_j) = \prod_{j\in\Z}\left(\prod_{i=0}^{k-1}\chi_i\circ\Delta_F (t,i-j)\right) (x_j).\]
First, we have 
\[j\in\support(\chi\circ F^t) \Rightarrow \exists i, \chi_i\circ\Delta_F (t,i-j) \neq 1 \Rightarrow \exists i\in \support(\chi), \Delta_F (t,i-j)\neq 0.\]
Therefore we get the upper bound $\rank(\chi\circ F^t)\leq \rank(\chi)\cdot d(t)$.

Second, if ${(t,-j)}$ is $k-1$-isolated, then $j\in\support({\chi\circ F^t})$. Indeed:
\[(\chi\circ F^t)_j = \prod_{i=0}^{k-1}\chi_i\circ\Delta_F (t,i-j) = \chi_0\circ\Delta_F(t,-j).\]
We deduce that ${s_{k-1}(t)\leq \rank(\chi\circ F^t)}$.
\end{proof}

\begin{example}
  By Propositions~\ref{prop:chidiffusive} and~\ref{prop:rankdeps}, having $s_k(t)\longrightarrow_t+\infty$ for all $k$ is a sufficient condition for randomization, but it is not necessary. For instance, take $F' = F\times(\sigma^N\circ F)$ where $F$ is any abelian CA that randomizes in density. By Corollary~\ref{coro:randomizingstability} below, $F'$ is randomizing in density. However, when $N$ is large enough that $F$ and ${\sigma^N\circ F}$ have disjoint neighborhoods, $F'$ has no bijective dependency, and therefore $s_k(t)=0$ for all $k$ and $t$.

  On the other hand, having many bijective dependencies is not enough if they are not well-isolated. For example, one can check that $H_2$ satisfies ${s_1(t)=t-2}$ but it is not randomizing, as shown in Section~\ref{sec:mainthm}.
\end{example}

\section{Duality, Diffusivity and Solitons}
\label{sec:dualitysolitons}

In the last section, we saw that the iterated images of characters under the action of a CA is key to understanding its action on probability measures. It turns out that the action of abelian CA on characters can be seen as a CA on the dual group, and that furthermore this dual CA shares many properties with the original CA. 

Remember that any character $\chi$ of $\dual{G^\Z}$ can be written as a finite product of cellwise (elementary) characters: ${\chi(x)=\prod_{z\in \Ne}\chi_z(x_z)}$ for some finite set $\Ne\subset \Z$ and $\chi_z\in\dual{G}$. To such a $\chi$ we associate $\confcar{\chi}$, the configuration of $\dual{G}^\Z$ defined by:
\[\confcar{\chi}(z) =
\begin{cases}
  \chi_z &\text{ if }z\in \Ne\\
  1 &\text{ else.}
\end{cases}
\]

Note that ${\confcar{\dual{G^\Z}}}$ is exactly the set of finite configurations of $\dual{G}^\Z$.

\begin{definition}[Dual CA]Let $F$ be an abelian CA over ${G^\Z}$. It can be written as
\[F(x)_z = \sum_{i\in \Ne} \alfaendo_i(x_{z+i})\] where ${\Ne\subset\Z}$ is finite and $\alfaendo_i$ are endomorphisms of $G$. We define $\dual{F}$ over the finite configurations of $\dual{G}^\Z$ by: 
\[\dual{F}(\confcar{\chi}) = \confcar{\chi\circ F}.\]
Since $\dual{F}$ is uniformly continuous and shift-invariant on finite configurations, it can be extended by continuity to a cellular automaton ${\dual{G}^\Z}\to {\dual{G}^\Z}$; this is the dual CA of $F$, and is an abelian CA for the group ${(\dual{G},\times)}$.  \end{definition}

More concretely, if ${\chi(x)=\prod_{z\in A}\chi_z(x_z)}$, we have:
\[\dual{F}(\confcar{\chi}) = 
\begin{cases}
  x_z\mapsto \prod_{i\in\Ne}\chi_{z-i}\bigl(\alfaendo_i(x_z)\bigr) &\text{ if }z\in A+\Ne\\
  1 &\text{ else.}
\end{cases}
\]

Then, for any ${c\in\dual{G}^\Z}$ we can define:
\begin{equation}
\dual{F}(c)_z = \prod_{i\in \Ne}\gamma_i(c_{z-i})\label{eq:defdual}
\end{equation}
where $\gamma_i$ is the endomorphism of $\dual{G}$ defined by:
\[\gamma_i(\chi) = g\mapsto \chi\circ\alfaendo_i(g).\]

When ${G=\F_p^d}$, the dual of a CA is obtained (up to conjugacy) by applying a mirror operation and transposing the matrix corresponding to each coefficient. Indeed, the map: 
\[\begin{array}{rcl}G&\to&\dual{G}\\a &\mapsto& \chi_a\end{array}\qquad \text{where}\qquad \chi_a:b\in\A \mapsto e^{\frac{2i\pi}{p}<a,b>},\]
where $<a,b>$ denotes the scalar product of $a$ and $b$ seen as $d$-dimensional vectors, is an isomorphism. Through that isomorphism, we have that ${\chi_a\circ M = \chi_{M^ta}}$ for any endomorphism $M : \A\to\A$, and the result comes from Equation (\ref{eq:defdual}).

In particular, our examples $F_2$ and $H_2$ are flip conjugate to their own dual since all their coefficients are symmetric matrices. $H_2$ is actually conjugate to its dual since it is left-right symmetric.

We do not know whether an abelian CA $F$ is always flip conjugate to its dual $\dual{F}$; however, we show in the remainder of this section that they are dynamically close enough that properties like randomization or diffusion are preserved by duality.

\begin{lemma}
  \label{lem:dualcompose}
  Let $\Phi_1$ and $\Phi_2$ be two abelian CA over ${G^\Z}$. Then we have ${\dual{\Phi_1\circ \Phi_2} = \dual{\Phi_2}\circ\dual{\Phi_1}}$. As a consequence:
  \begin{itemize}
  \item ${\dual{F^t} = \bigl(\dual{F}\bigr)^t}$ for any $t>0$,
  \item ${\dual{F\circ\sigma} = \dual{F}\circ\sigma^{-1}}$,
  \item If $F$ is reversible, then $\dual{F}$ is also reversible and ${\dual{F^{-1}} = \bigl(\dual{F}\bigr)^{-1}}$.
  \end{itemize}
  Furthermore, $\dual{\dual{F}} = F$ up to a canonical isomorphism.
\end{lemma}
\begin{proof}
  By definition of dual CAs, we have for any ${\chi\in\dual{G^\Z}}$:
  \[\dual{\Phi_1\circ \Phi_2}(\confcar{\chi}) = \confcar{\chi\circ \Phi_1\circ \Phi_2} = \dual{\Phi_2}(\confcar{\chi\circ \Phi_1}) = \dual{\Phi_2}\circ\dual{\Phi_1}(\confcar{\chi}).\]
  Since ${\confcar{\dual{G^\Z}}}$ is dense in ${\dual{G}^\Z}$ we deduce that ${\dual{\Phi_1\circ \Phi_2} = \dual{\Phi_2}\circ\dual{\Phi_1}}$ on the whole space.

  For the last point, it is well-known that $\dual{\dual{G^\Z}} \simeq G^\Z$ through the canonical isomorphism $\psi:g\mapsto(\chi\mapsto \chi(g))$ (see e.g. \cite{Deitmar}, Lemma 4.1.4). Then we check that $\dual{\dual{F}} : \psi(g) \mapsto \psi(g) \circ \dual{F} = \psi(F(g))$, so that $\dual{\dual{F}}\simeq F$ up to this isomorphism. 
\end{proof}

In the following two technical lemmas and in the remainder of the section, we stress when our results do not require the CA to be abelian, even though we will only apply them to abelian CA.

\begin{lemma}
  \label{lem:rankincrease}
  Let $F$ be a CA with quiescent state $0$. Then for any finite configuration $x$, $\rank(F(x)) \leq |\Ne|\cdot\rank(x)$.

  In particular, if $F$ is abelian and $\chi$ is a character , $\rank(\chi\circ F)\leq |\Ne|\cdot\rank(\chi)$.
\end{lemma}

\begin{proof}
Since $0$ is quiescent, the only nonzero cells in $F(x)$ belong to $\support(x)+\Ne$. It follows that $\rank(F(x)) \leq |\Ne|\cdot \rank(x)$. The second statement follows by applying the result to $\dual{F}$, noticing that $F$ and $\dual{F}$ have the same neighbourhood size by Equation~\ref{eq:defdual}.
\end{proof}

The converse lemma holds for reversible CA:

\begin{lemma}
  \label{lem:rankdecreasing}
  Let $F$ be a reversible CA with quiescent state $0$. There exists a constant $C > 0$ such that for any finite configuration $x$, $\rank(F(x)) \geq C\cdot\rank(x)$.

  In particular, if $F$ is abelian, there exists $C>0$ such that $\rank(\chi\circ F) \geq C \cdot \rank(\chi)$ for any character $\chi$.
\end{lemma}
%{eq:defdual}
\begin{proof}Apply Lemma~\ref{lem:rankincrease} on $F^{-1}$. The second point uses the last point of Lemma~\ref{lem:dualcompose}. \end{proof}

\begin{definition}
  A CA $F$ with a quiescent state $0$ is \emph{strongly diffusive} (resp. diffusive in density) if for any finite configuration $c$, we have ${\rank\bigl(F^t(c)\bigr)\to \infty}$ (resp. on a subsequence of density $1$).
\end{definition}

\begin{lemma}
  An abelian CA $F$ is strongly \chidiffusive (resp. \chidiffusive in density) if and only if $\dual{F}$ is strongly diffusive (resp. diffusive in density).
\end{lemma}
\begin{proof}
  ${\dual{F}(\confcar{\chi}) = \confcar{\chi\circ F}}$ and ${\rank(\chi) = \rank\bigl(\confcar{\chi}\bigr)}$.
\end{proof}

\newcommand\s{\sigma}

\begin{definition}[Soliton] Let $F$ be an abelian CA. A \emph{soliton} is a finite configuration ${c\neq\overline{0}}$ such that ${F^p(c) = \sigma^q(c)}$ for some ${p \geq 1}$ and ${q\in\Z}$.
  \end{definition}

Intuitively, having a soliton is the opposite of being diffusive (even in density). In the remainder of the section we will develop this intuition and prove a series of technical results about solitons that will be culminate in the characterization of randomization in density in the next section.

  First note that all the configurations in the orbit of a soliton have bounded rank. Conversely, we can extract a soliton from any orbit of finite configurations whose rank is bounded on a set of time steps of positive density:

\begin{proposition}\label{prop:solitons}
  Let $F : \az\to\az$ be a surjective cellular automaton with a quiescent state 0. Assume that $F$ is not diffusive in density. Then $F$ admits a soliton.
\end{proposition}

\begin{proof}
There is a finite initial configuration $x$ and an increasing sequence $(T_n)_{n\in\N}$ of positive upper density such that $\rank(F^{T_n}(x))$ is bounded.
  %Denote such times $t$ \emph{nice times}.
  Without loss of generality, assume that $\rank(F^{T_n}(x))=k$ for all $n$. Denote
  $i_1(T_n),\dots, i_k(T_n)$ the nonzero coordinates at time $T_n$.

  Now let $m\in\{1,\dots,k\}$ be the maximum integer such that there exists an integer
  $M$ such that the subsequence:
  \[(T_{\varphi(n)})_{n\in\N} = \left\{n\in\N : i_m(T_n) - i_1(T_n) \leq M\right\}\]
  has positive upper density. In particular $(T_{\varphi(n)})_{n\in\N}$ has positive upper density. We distinguish two cases.

  \fbox{$m=k$:} $F^{T_{\varphi(n)}}(x)_{[i_1(T_{\varphi(n)}),i_k(T_{\varphi(n)})]}$
  can take at most $|\bigcup_{j=0}^M\A^j| = \sum_{j=0}^M|\A|^j$
  different values. By the pidgeonhole principle, we can find two
  integers $a < b$ such
  that
  \[F^{T_{\varphi(a)}}(x)_{[i_1(T_{\varphi(a)}),i_k(T_{\varphi(a)})]}
  =
  F^{T_{\varphi(b)}}(x)_{[i_1(T_{\varphi(b)}),i_k(T_{\varphi(b)})]}.\]
  But this means that
  $F^{T_{\varphi(b)}}(x) = \s^{i_1(T_{\varphi(b)})-i_1(T_{\varphi(a)})} \circ
  F^{T_{\varphi(a)}}(x)$, so we have found a soliton.

  \fbox{$m<k$:} First, by the same argument as above,
  $F^{T_{\varphi(n)}}(x)_{[i_1(T_{\varphi(n)}),i_m(T_{\varphi(n)})]}$
  can only take a finite number of values, so at least one of these
  words appear with positive density. Denote $w$ the corresponding
  word and $(T_{\varphi'(n)})_{n\in\N}$ the corresponding subsequence.
  We now prove that the configuration
  \[\cdots 0\cdot 0\cdot 0\cdot w\cdot 0\cdot 0\cdot 0\cdots \]
  is a soliton.

  Take $N\in\N$ such that $\frac 1N$ is a lower bound on the density of $(T_{\varphi'(n)})_{n\in\N}$ and let $r$ be the radius of $F$. By construction of $m$, the times $T_{\varphi'(n)}$ where $i_{m+1}(T_{\varphi'(n)})-i_m(T_{\varphi'(n)}) \leq 2rN$ have upper density $0$. Therefore we extract from the sequence $(T_{\varphi'(n)})_{n\in\N}$ a new subsequence $(T_{\varphi''(n)})_{n\in\N}$ corresponding to times $t$
  where $i_{m+1}(t)-i_m(t) > 2rN$ with the same upper density. In particular,
  we can find some $n$ such that $T = T_{\varphi''(n+1)} - T_{\varphi''(n)} \leq N$.

  As shown in Figure~\ref{fig:soliton}, only two disjoint areas can contain nonzero values in $F^{T_{\varphi''(n+1)}}(x)$:
  \begin{itemize}
  \item the interval
    $[i_1(T_{\varphi''(n)})-rT,i_m(T_{\varphi''(n)})+rT]$,
    
    which contains $f^T\left(0^{2rT}\cdot w\cdot 0^{2rT}\right)$;
  \item the interval
    $[i_{m+1}(T_{\varphi''(n)})-rT,i_k(T_{\varphi''(n)})+rT]$

    which contains
    $f^T\left(0^{2rT}\cdot
      F^{T_{\varphi''(n)}}(x)_{[i_{m+1}(T_{\varphi''(n)}),i_k(T_{\varphi''(n)})]}\cdot
      0^{2rT}\right)$.
  \end{itemize}

  \begin{figure}
    \centering
    \includegraphics{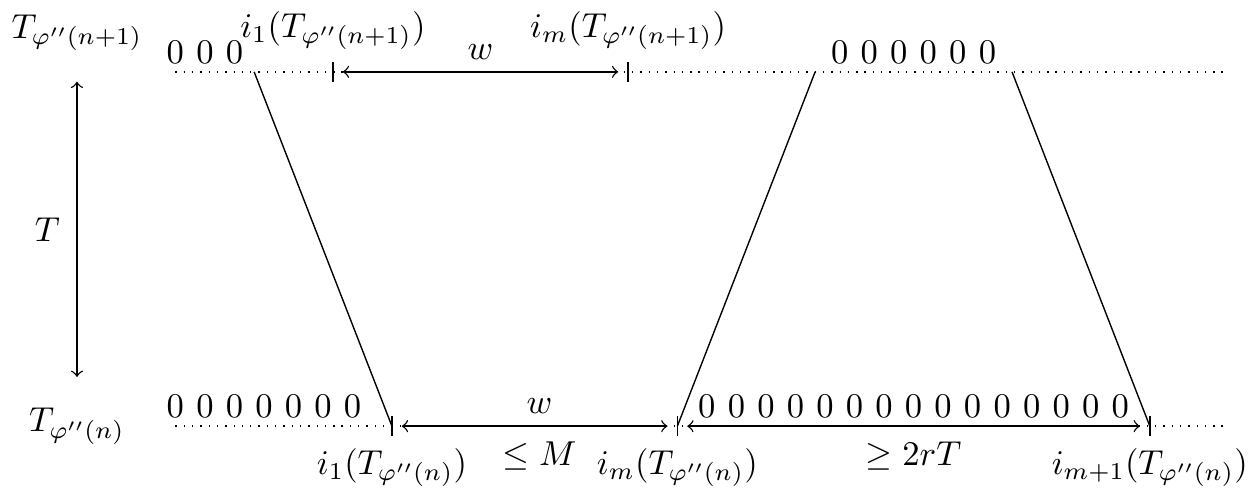}
    \caption{The finite configuration defined by $w$ is a soliton.}
    \label{fig:soliton}
  \end{figure}
  
  Indeed, any cell outside of these regions can be written as
  $f^T(0^{2rT+1}) = 0$ since $0$ is quiescent. Consider the
  different possibilities for the value of $i_1(T_{\varphi''(n+1)})$:

  \begin{itemize}
  \item if $i_1(T_{\varphi''(n+1)}) < i_1(T_{\varphi''(n)})-rT$, then
    $F^{T_{\varphi''(n+1)}}(x)_{i_1(T_{\varphi''(n+1)})}=f^T(0^{2rT+1})=0$ which is a
    contradiction;
  \item if $i_1(T_{\varphi''(n+1)}) > i_1(T_{\varphi''(n)})+rT$, then
    $F^{T_{\varphi''(n)}}(x)_{[i_1(T_{\varphi''(n)})\pm rT]}$ is a
    nonzero word whose image under $f^T$ is zero. This means $F$ is not preinjective (two different configurations that differ on a finite subset of cells have same image), so by the Garden-of-Eden theorem \cite{hedlund,Coornaert10} it is not surjective, a contradiction.
  \end{itemize}

  Therefore we have
  $i_1(T_{\varphi''(n+1)}) \in [i_1(T_{\varphi''(n)})\pm rT]$. In
  particular, the interval
  $I = [i_1(T_{\varphi''(n)})-rT,\ i_m(T_{\varphi''(n)})+rT]$ contains all
  $i_\ell(T_{\varphi''(n+1)})$ for $\ell \leq m$. Using a similar argument,
  $i_{m+1}(T_{\varphi''(n+1)}) \in [i_{m+1}(T_{\varphi''(n)})\pm rT]$, so that $I$
  does not contain any $i_\ell(T_{\varphi''(n+1)})$ for
  $\ell > m$.

  From this we conclude that for some constant $C$:
  \begin{align*}
%    F^T\left(0^{2rT}\cdot F^{T_{\varphi''(n)}}(x)_{I}\cdot 0^{2rT}\right) &= 0^{C}\cdot F^{T_{\varphi''(n)}}(x)_{I}\cdot 0^{2rT-C}\\
    F^T\left(0^{2rT}\cdot w\cdot 0^{2rT}\right) &= 0^{C}\cdot w\cdot 0^{2rT-C}
  \end{align*}
  and therefore we have found a soliton.
\end{proof}

The remainder of this section is dedicated to proving the following proposition.

\begin{proposition}
  \label{prop:solitondual}
  Let $F$ be an abelian CA. Then $F$ has a soliton if and only if $\dual{F}$ has a soliton.
\end{proposition}

To prove this proposition, we need a series of lemmas. For an abelian CA $F$, a finite fixed point is just a fixed point which is also a finite configuration. A finite fixed point is non-trivial if it is not the configuration everywhere equal to $0$.

\begin{lemma}
  \label{lem:fixper}
  Let $F$ be an abelian CA and denote by $X_F$ the set of spatially periodic fixed points: \[{X_{F,n} = \{x : \sigma^n(x)=x\text{ and }F(x)=x\}} \qquad X_F = \bigcup_n X_{F,n}.\]
  $F$ has a non-trivial finite fixed point if and only if $X_F$ is infinite. In this case we actually have $|X_{F,n}| = 2^{\Omega(n)}$.
\end{lemma}
\begin{remark}
  This remark holds in dimension $d$ if one replaces the infiniteness assumption by $|X_{F,n}| = 2^{\Omega(n^d)}$.
  \end{remark}

\begin{proof}
  First suppose that ${F(x)=x}$ where $x$ is a nontrivial finite configuration, and denote by $u$ a finite word containing the nonzero part of $x$. Let $r$ be the radius of $F$ and $k = |u|+2r$. Consider the set of finite words:
\[W_n = \left\{w\in\A^n\ :\ 
\begin{array}{ll}&\forall i, 0\leq i<\left\lfloor\frac nk\right\rfloor\Rightarrow \ w_{[ki,k(i+1)-1]} = 0^k \mbox{ or } w_{[ki,k(i+1)-1]} = 0^ru0^r\\
\mbox{and }&\\
&w_{k\cdot\left\lfloor\frac nk\right\rfloor, n} = 0\end{array}
\right\}.\]

Any periodic configuration made of concatenated copies of some word in $W_n$ (excepting $0^n$) is a nontrivial fixed point and $|W_n| = 2^{\lfloor\frac nk\rfloor}$. Therefore we have ${|X_{F,n}|} = 2^{\Omega(n)}$.

Conversely, suppose $X_F$ is infinite. Then, since $X_{F,n}\subset X_{F,kn}$ for $k\in\N^+$, $n\mapsto |X_{F,n}|$ is not bounded from above. For some $n$, $X_{F,n}$ must contain at least $2$ distinct configurations $x_1$ and $x_2$ such that $x_1|_{[1,r]} = x_2|_{[1,r]}$ and $x_1|_{[n-r+1,n]} = x_2|_{[n-r+1,n]}$ by the pidgeonhole principle. It follows that the configuration 
\[x : z\in\Z\mapsto
\begin{cases}
  0 &\text{ if ${z\leq0}$ or ${z>n}$}\\
  x_1(z)-x_2(z) &\text{ else}
\end{cases}
\]
is a non-trivial finite fixed point.
\end{proof}

\begin{lemma}
  \label{lem:dualkernel}
  Let $G$ be an abelian group. For any endomorphism ${h} : G \to G$, define its dual ${\dual{h}}$ by ${\dual{h}(\chi) = \chi\circ h}$ for any character ${\chi\in\dual{G}}$. Then we have ${|\ker(h)|=|\ker(\dual{h})|}$.
\end{lemma}

\begin{proof}
We have $\chi\in\ker(\dual h) \Leftrightarrow \chi\circ h = 1 \Leftrightarrow \Ima(h) \subset \ker(\chi)$. Therefore the restriction
\[\chi \mapsto \chi|_{(G/\Ima h)}\]
is a bijection between $\ker(\dual h)$ and $\dual{(G/\Ima h)}$. Since $\left |\dual{(G/\Ima h)}\right| = \left|G/\Ima h\right| = |\ker h|$ by Proposition~\ref{prop:isodual}, we conclude. 
\end{proof}

\begin{lemma}
  \label{lem:dualfixedpoint}
  Let $F$ be an abelian CA over alphabet $G$. $F$ has a non-trivial finite fixed point if and only if $\dual{F}$ has a non-trivial finite fixed point.  
\end{lemma}
\begin{proof}
  For any ${n \in\N}$ with $n \geq 1$, let $h_n$ be the endomorphism of ${G^n}$ defined as follows. For any $u\in G^n$, let $x^u\in G^\Z$ be the (spatially) periodic point of period $u$, i.e. $x^u_i = u_{i\mod n}$ for all $i$. For each $u \in G^n$, define
  \[ h_n(u) = \left(F(x^u) - x^u\right)_{[0,n-1]}.\] ${h_n}$ captures the action of the CA ${F-Id}$ on spatially periodic configurations of period $n$. In particular we have ${|\ker(h_n)|=|X_{F,n}|}$. Now consider its dual $\dual{h_n}$, which is an endomorphism of ${\dual{G^n}} = (\dual{G})^n$. For any $\chi = \prod_{1\leq i\leq n}\chi_i$,

  \begin{align*}
    \dual{h_n}(\chi): u\mapsto \prod_{i=1}^n\chi_i(F(x^u)-x^u)_i &= \prod_{i=1}^n\chi_i\circ F(x^u)_i\cdot (\chi_i(x^u)_i)^{-1}\\
    &= (\dual{F}(x^\chi)\cdot x^{1/\chi}) (x^u)
  \end{align*}  
  by definition of $\dual{F}$, and where again we define the periodic point $x^\chi\in\dual{G}^\Z$ by $x^\chi_i = \chi_{i\mod n}$. Therefore we have ${|\ker(\dual{h_n})|=|X_{\dual{F},n}|}$, similarly as above.

  %, one can check that ${\dual{h_n}}$ is conjugated to $g_n$, the endomorphism of ${\dual{G}^n}$ defined by 
  %\[g_n(\chi_1,\ldots,\chi_n) = \bigl(\dual{F}(x_\chi)(1) / x_\chi(1),\ldots, \dual{F}(x_\chi)(n) / x_\chi(n)\bigr)\] where ${x_\chi(z) = \chi_{z\bmod n}}$.

Now, by Lemma~\ref{lem:dualkernel}, we have ${|\ker(h_n)|=|\ker(\dual{h_n})|}$ and therefore ${|X_{F,n}|=|X_{\dual{F},n}|}$. We conclude by Lemma~\ref{lem:fixper}.
\end{proof}

\begin{proof}[of Proposition~\ref{prop:solitondual}]
  Suppose $F$ has a soliton $x$, that is, ${F^p(x) = \sigma^q(x)}$ for ${p \geq 1}$ and ${q\in\Z}$. Then $x$ is a finite fixed point for the abelian CA ${\sigma^{-q}\circ F^p}$. By Lemma~\ref{lem:dualfixedpoint} we deduce that ${\dual{\sigma^{-q}\circ F^p}}$ also has a finite fixed point $y$. Using Lemma~\ref{lem:dualcompose}, we can rewrite this as ${\bigl(\dual{F}\bigr)^p(y) = \sigma^{-q}(y)}$, which shows that $y$ is actually a soliton of $\dual{F}$.

Applying the same reasoning, a soliton in ${\dual{F}}$ implies a soliton in $\dual{\dual{F}}$, and $\dual{\dual{F}} = F$ up to a canonical isomorphism by Lemma~\ref{lem:dualcompose}.
\end{proof}

\begin{example}
  Note that the smallest solitons of $F$ and $\dual{F}$ need not be
  of the same size. For example, consider the CA $F$ defined over the alphabet $\F_2^2$ defined by:
  \[F(x)_z =
  \begin{pmatrix}
    1 & 1\\
    0 & 1
  \end{pmatrix}\cdot x_z +
  \begin{pmatrix}
    0 & 0\\
    0 & 1
  \end{pmatrix}\cdot x_{z+1} .\]
  $F$ acts like the identity on any configuration whose second
  components are all $0$. In particular it admits solitons of
  rank $1$. However, its dual, obtained up to conjugacy by mirroring and transposing the coefficients:
  \[\dual{F}(x)_z =
  \begin{pmatrix}
    0 & 0\\
    0 & 1
  \end{pmatrix}\cdot x_{z-1}+ \begin{pmatrix}
    1 & 0\\
    1 & 1
  \end{pmatrix}\cdot x_z
  \]
  has no soliton of rank $1$. Indeed, take any finite configuration
  $x$ supported in $[i,j]$ for $i\leq u$ and such that $x_i\neq 0, x_j\neq 0$. Notice that:
  \[F(x)_i = \left(\begin{matrix} (x_i)_1\\(x_i)_1+(x_i)_2\end{matrix}\right)\neq 0\qquad\text{and}\qquad F(x)_{j+1} = \left(\begin{matrix} 0\\(x_j)_2\end{matrix}\right).\]
In particular, if $(x_j)_2 \neq 0$, then $x$ cannot be a soliton.
  
Now take any finite $x$ of rank $1$, assuming $x_0\neq 0$. If $(x_0)_2 \neq 0$ then $x$ is not a soliton by the previous argument. If $(x_0)_2 = 0$ then $\dual{F}(x)$ is a finite configuration of rank $1$ such that $\dual{F}(x)_0 = \left(\begin{matrix}(x_0)_1\\(x_0)_1\end{matrix}\right)$ and is not a soliton.

Finally, it is easy to check that the configuration
  \[y = \cdots\begin{pmatrix}0\\0\end{pmatrix}\begin{pmatrix}0\\1\end{pmatrix}\begin{pmatrix}1\\0\end{pmatrix}\begin{pmatrix}0\\0\end{pmatrix}\cdots\]
  is a fixed point, and therefore a soliton, of ${\dual{F}}$.
\end{example}
\section{Characterization of Randomization in Density}
\label{sec:mainthm}

% \begin{definition}
% Let $f : S^\Z \to S^\Z$ be a CA with one-sided radius $r$. Then $f$ \emph{uniformly randomizes} a class of measures $\mathcal{M} = \bigcup_{j \in \N} \mathcal{M}_j$ where $\mathcal{M}_j \subset \mathcal{M}(S^{[0,j-1]})$ if for all $\ell \in \N$ and $\epsilon > 0$, there exists $N$ such that whenever $n \geq N$, the map $f^n : S^{\ell+nr} \to S^{\ell}$ satisfies
% \[ |P(f^n(u) = w \;|\; u \sim \mu) - 1/|S^{\ell}|| \leq \epsilon \]
% for all $w \in S^\ell$ and $\mu \in \mathcal{M}_{\ell+nr}$.
% \end{definition}
%

%\begin{definition}
%Let $\M = \bigcup_j\M_j$ be a class of probability measures. $F$ randomizes $\M$ uniformly in $j$ if there exists a family of functions $f_j : \N\to \R$ such that :
%\begin{itemize}
%  \item $\forall j, f_j(t) \underset{t\to\infty}\longrightarrow 0$;
%  \item $\forall j, \forall \mu\in\M_j, d(F^t\mu, \lambda) \leq f_j(t)$.
%\end{itemize}
%\end{definition}

%
%\begin{definition}
%We say $f : S^\Z \to S^\Z$ is \emph{nice} if it randomizes the class of (nice?) hidden Markov random fields uniformly in their size and weight.
%\end{definition}

We are now ready to prove our first main result, which is a combinatorial characterization of randomization in density by the absence of solitons.
\begin{theorem}
  \label{thm:main}
  Let $F$ be an abelian CA. The following are equivalent:
  \begin{enumerate}[(i)]
  \item $F$ randomizes in density any harmonically mixing measure;
  \item $F$ has no soliton;
  \item $F$ is diffusive in density;
  \item For some strongly nonuniform Bernoulli measure $\mu$, the sequence $(F^t\mu)_{t\in\N}$ admits $\lambda$ as an accumulation point.
  \end{enumerate}
\end{theorem}

\begin{proof}
Write the dual claims of $(i), (ii), (iii), (iv)$ for $\hat F$ as $(j), (jj), (jjj), (jv)$, respectively. Then
\[ (i) \overset{Prop.~\ref{prop:chidiffusive}}{\iff} (jjj) \overset{Prop.~\ref{prop:solitons}}{\iff} (jj) \overset{Prop.~\ref{prop:solitondual}}{\iff} (ii) \overset{Prop.~\ref{prop:solitons}}{\iff} (iii), \]
so the first three claims are equivalent, and are also equivalent to $(jj)$. Moreover, it is clear that $(i)\Rightarrow (iv)$. We prove $(iv) \Rightarrow (jj)$. Suppose $(jj)$ does not hold, and let $\chi \in \dual{G}^\Z$ be a soliton for $\dual{F}$. Solitons are finite nonzero configurations, so we can consider $\chi \in \dual{G^\Z}$ to be a nontrivial character satisfying that $\rank(\dual{F}^t(\chi))$ is bounded from above by some $m$. Since $\mu$ is strongly nonuniform and Bernoulli, there exists $\varepsilon > 0$ such that $\mu[\chi'] \geq \varepsilon^{\rank{\chi'}}$. In particular, ${F^t\mu[\chi]} = \mu[\dual{F}^t(\chi)] \geq \varepsilon^m$ for all $t$. Since $\lambda[\chi] = 0$, ${F^t\mu[\chi]}$ does not have $\lambda$ as an accumulation point.

 % The equivalence $(i)\Leftrightarrow(iii)$ comes from Propositions~\ref{prop:chidiffusive} and \ref{prop:solitondual}, the implication $(ii)\Rightarrow(iii)$ is Proposition~\ref{prop:solitons} and its converse is obvious. Moreover, it is clear that $(i)\Rightarrow (iv)$. We prove $(iv)\Rightarrow (ii)$. Let us suppose $(iv)$ holds. Then for any non-trivial character $\chi$ we can make ${F^t\mu[\chi]} = \mu[\dual{F}^t(\chi)]$ arbitrarily close to ${0=\lambda[\chi]}$ by an appropriate choice of $t$. However, as shown in the proof of Proposition~\ref{prop:chidiffusive}, the fact that $\mu$ is strongly nonuniform implies that there is a constant $m$ such that, for any non-trivial character $\chi$, ${\mu[\chi]\geq m^{\rank(\chi)}}$. Therefore $\rank(\dual{F}^t(\chi))$ is not bounded from above, which means that $\dual{F}$ has no soliton. By Proposition~\ref{prop:solitondual}, we are done.
\end{proof}

Before giving some general consequences of this theorem and applying it to the commutative case, let's use it on our examples.
\begin{example}
  The CA $H_2$ admits a soliton:
  \[\cdots\begin{pmatrix}0\\0\end{pmatrix}\begin{pmatrix}0\\1\end{pmatrix}\begin{pmatrix}1\\0\end{pmatrix}\begin{pmatrix}0\\0\end{pmatrix}\cdots\]
        and therefore it is not randomizing in density.
        
        On the contrary, we show that $F_2$ has no soliton and therefore is randomizing in density (it is actually strongly randomising, see Section~\ref{sec:strongrandom}). This is an alternative proof to the result of \cite{Maass1999} where it was proven through a delicate analysis of $F_2$ using binomial coefficients and Lucas' Lemma. Our proof is illustrated in Figure~\ref{img:nosoliton}. It is easy to show by induction that: 
        \[\forall c\in\az, \forall n\in\N, \forall z\in\Z,\ F_2^{2^{n+1}}(c)_z+F_2^{2^{n}}(c)_{z+2^n}+c_z = 0\bmod 2.\]
        \begin{center}
          \begin{figure}[!h]
            \begin{tabular}{cc}
              \phantom{\hspace{.5cm}}\includegraphics{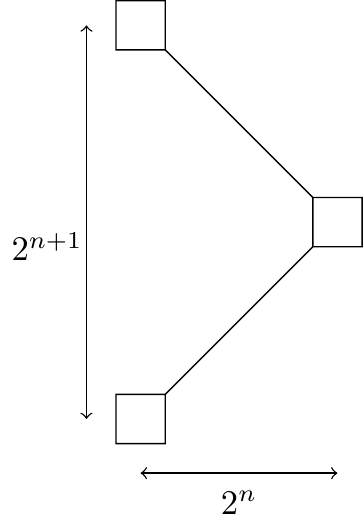}\phantom{\hspace{1cm}}& \includegraphics{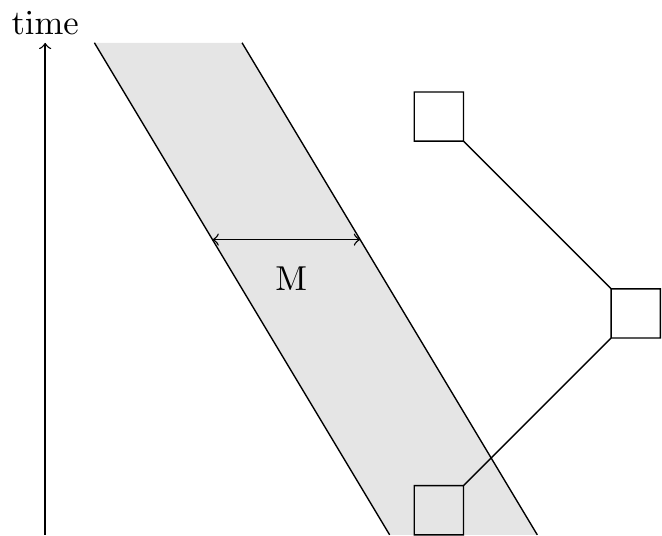}\\
              (a)&(b)\end{tabular}
            \caption{The direction of time is upward. (a) Any three cells in these relative positions in the space-time diagram of $F_2$ sum to zero (cancelling shape). (b) The space-time diagram of a soliton is zero except around a space-time line, and we can always position a cancelling shape that does not sum to zero.}\label{img:nosoliton}
          \end{figure}
        \end{center}
        %\TODO{Maybe here a figure}
        
        Suppose by contradiction that $F_2$ has a soliton $c$ such that $F^p(c)=\sigma^q(c)$. Therefore there is a constant $M$ such that any nonzero cell of the space-time diagram $(F^t(c))_{t\in\N}$ is at horizontal distance at most $M$ of the real line ${L = \{(z,t) : pz+qt = 0\}}$. In other words, $F^t(c)_z\neq 0 \Rightarrow |pz+qt|<pM$. Now take $n$ such that $2^n>2M$ and any $|z|\leq M$ and distinguish three cases:
        \begin{description}
        \item[$q=0$:] Since $|z-2^n| > M$, we have $c_{z-2^n}=0$ and $F_2^{2^{n+1}}(c)_{z-2^n} = 0$, so that $F_2^{2^{n}}(c)_{z}=0$. This is true for every $z$ such that $|z|\leq M$, so $F_2^{2^{n}}(c) = 0$, a contradiction.
        \item [$q=-p$:] Assume $p=-q=1$. $c_{z+2^{n+1}} = 0$ and $F_2^{2^n}(c)_{z+2^{n+1}+2^n} = 0$, so that $F_2^{2^{n+1}}(c)_{z+2^{n+1}}=0$. At time $2^{n+1}$, this applies to every $z$ such that $|p(z+2^{n+1})-q(2^{n+1})| = |z|\leq M$, so $F_2^{2^{n+1}}(c) = 0$, a contradiction.
        \item[otherwise,] $F_2^{2^{n+1}}(c)_z = F_2^{2^{n}}(c)_{z+2^n} = 0$ when $n$ is large enough, so that $c_z = 0$ for all $|z|\leq M$, a contradiction.
        \end{description}

\end{example}

\begin{definition}[Positive expansiveness]A CA is \emph{positively expansive} if there is some finite ${W\subseteq\Z}$ such that for any pair of distinct configurations ${x,y\in\az}$:
  \[\exists t\in\N,\ \exists z\in W,\ {F^t(x)_z\neq F^t(y)_z}.\]
  More generally, for ${\alpha\in\R}$, we say that $F$ is \emph{positively expansive in direction $\alpha$} if there is some finite ${W\subseteq\Z}$ such that for any pair of distinct configurations ${x,y\in\az}$,
    \[\exists t\in\N,\ \exists z\in W,\ {F^t(x)_{z+\lceil\alpha t\rceil}\neq F^t(y)_{z+\lceil\alpha t\rceil}}\]
\end{definition}
 See \cite{Sablik08,DelacourtPST11} for further developments on directional dynamics in cellular automata.

 In the next result, for a CA $F : \az \to \az$ and a subalphabet $\mathcal B\subset\A$ such that $F(\mathcal B^\Z)\subseteq \mathcal B^\Z$, the corresponding \emph{subautomaton} of $F$ is $F' = F|_{\mathcal B^\Z}$.

\begin{corollary}
  \label{coro:randomizingstability}
  Let $F$ and $G$ be abelian CA. Then:
  \begin{itemize}
  \item if $F$ and $G$ are randomizing then so is ${F\times G}$;
  \item if $F$ is randomizing then all its subautomata are;
  \item if $F$ is randomizing and reversible, then so is $F^{-1}$;
  \item if $F$ has a direction of positive expansivity then it is randomizing;
  \end{itemize}
  where randomizing means randomizing in density any harmonically mixing measure.
\end{corollary}
\begin{proof}
  This corollary follows from Theorem~\ref{thm:main} by the following elementary observations on solitons:
  \begin{itemize}
  \item a soliton in ${F\times G}$ implies a soliton in either $F$ or $G$;
  \item a soliton in a subautomaton of $F$ is a soliton for $F$;
  \item a soliton for $F^{-1}$ is a soliton for $F$;
  \item a positively expansive CA cannot admit a soliton. Indeed, take a CA $F$ with a direction of positive expansiveness $\alpha$ and assume for the sake of contradiction that it admits a soliton $c$: $F^p(c) = \s^q(c)$. For any finite $W\subset \Z$, take the two distinct configurations $x=0$ and $y=\s^k(c)$ for $|k|$ large enough and ${\textrm{sign}(k)=\textrm{sign}(q/p-\alpha)}$, and check that ${F^t(x)_{z+\lceil\alpha t\rceil}= F^t(y)_{z+\lceil\alpha t\rceil} = 0}$ for every $t\in\N$ and $z\in W$.
  \end{itemize}\end{proof}

\begin{remark}
  A CA with local rule ${f:\A^m\rightarrow\A}$ is bipermutive if $m \geq 2$ and the maps ${x\mapsto f(x,a_1,\ldots,a_{m-1})}$ and ${x\mapsto f(a_1,\ldots,a_{m-1},x)}$ are permutations of $\A$ for any ${a_1,\ldots,a_{m-1}\in\A}$. Since bipermutivity implies the existence of a direction of positive expansivity \cite{Cattaneoetal}, the above corollary implies that any bipermutive abelian CA randomizes in density any harmonically mixing measure. 
This generalizes Theorem 9 of \cite{PivatoYassawi2}, where the authors consider abelian CA of the form 
  \[F = \sum_{i\in\Ne}\confendo_i\circ\sigma^i\]
  where ${|\Ne|\geq 2}$ and $\confendo_i$ are commuting automorphisms. We do not need this hypothesis here, and for instance we prove that the following CA over ${\F_2^2}$ is randomizing in density: 
  \[F(c)_z =
\begin{pmatrix}
        1 & 1\\
        0 & 1
      \end{pmatrix}\cdot c_z + 
      \begin{pmatrix}
        1 & 0\\
        1 & 1
      \end{pmatrix}\cdot c_{z+1}.\]
\end{remark}

\section{Other forms of randomization}\label{sec:other}
In this section we consider other forms of randomization that have been less studied in the literature. First we prove that, in the case of abelian CAs whose coefficients are commuting endomorphisms, only randomization in density can happen. Then we provide examples of abelian CA that exhibit strong randomization and randomization for cylinders up to some fixed length.

\subsection{Abelian CAs with commuting coefficients}

The case of abelian CAs with commuting coefficients is in many regards similar to the case of scalar coefficients. These CAs have more rigidity in their time evolution than general abelian CAs: the image of a single cell at time $t$ can be determined directly through the use of the binomial theorem and modular arithmetic of binomial coefficients. In particular, when $t$ is some power of the order of the group, the number of bijective dependencies is bounded, which explains why these CAs cannot randomize strongly.

\begin{lemma}
  \label{lem:binomialeries}
  Let $p$ be a prime number and $l\geq 0$.  Let
  $(\mathcal{X},+,\times)$ be a commutative ring of characteristic
  $p^l$, \textit{i.e.} such that for any $X\in \mathcal{X}$:
  \[p^l\cdot X = 0\] where $0$ is the neutral element for $+$.

  Then, for any $n\geq 0$ and any elements
  $X_i\in\mathcal{X}$ , ${1\leq i\leq k}$, we have this equality in $\mathcal{X}$:
  \[\left(\sum_{i=1}^kX_i\right)^{p^{n+l-1}} =  \left(\sum_{i=1}^kX_i^{p^n}\right)^{p^{l-1}}\]
\end{lemma}

\begin{proof}
  First by the binomial theorem we have:
  \[\left(\sum_{i=1}^kX_i\right)^{p^n} = \left(X_1+\sum_{i=2}^kX_i\right)^{p^n} = X_1^{p^n} +
  \left(\sum_{i=2}^kX_i\right)^{p^n} + p \cdot Y\] for some $Y$ because, by Kummer's theorem, $p$ divides ${p^n \choose i}$ for ${0<i<p^n}$.  By a direct
  induction we deduce:
  \[\left(\sum_{i=1}^kX_i\right)^{p^{n}} = \sum_{i=1}^kX_i^{p^n} +
  p \cdot Y'\] for some $Y'$.  Now, applying again the binomial theorem
  we get:
  \begin{align*}
    \left(\sum_{i=1}^kX_i\right)^{p^{n+l-1}} &=
    \left(\sum_{i=1}^kX_i^{p^n}
      + p \cdot Y'\right)^{p^{l-1}}\\
    &= \left(\sum_{i=1}^kX_i^{p^{n}}\right)^{p^{l-1}} + \sum_{j=1}^{p^{l-1}} {{p^{l-1}}\choose j}\cdot p^{j} \cdot (Y')^j \cdot \left(\sum_{i=1}^kX_i^{p^{n}}\right)^{p^{l-1}-j}\\
    &= \left(\sum_{i=1}^kX_i^{p^{n}}\right)^{p^{l-1}} + p^{l-1} \cdot Z
  \end{align*}
  because by Kummer's theorem $p^l$ divides ${p^{l-1}\choose j}\cdot p^{j}$ for any $1\leq j\leq p$. Then the desired equality is shown since ${p^l\cdot Z' =0}$.
\end{proof}

Now we prove that abelian CA cannot randomize strongly, and cannot randomize some cylinders without randomizing in density.

\begin{theorem}
  \label{thm:commutingcase}
  There is no strongly randomizing abelian CA with commuting endomorphisms. Moreover, for any abelian group $G$, there exists an ${N\in\N}$ such that, for any abelian CA $F$ over $G$ with commuting endomorphisms, the following are equivalent:
  \begin{enumerate}[(i)]
  \item $F$ randomizes in density;
  \item ${\rank\bigl(F^N(c)\bigr)\geq 2}$ for any finite configuration
    $c$ of rank $1$;
  \item $F$ randomizes in density on cylinders of length $1$,
  \end{enumerate}
  where as usual the class of initial measures is the set of harmonically mixing measures.
\end{theorem}

\begin{proof}
  If $p$ is a prime number, a \emph{$p$-group} is a group $G'$ where the order of every element $g \in G'$ is a power of $p$. By the decomposition theorem for finite abelian groups, they can be written as a direct product of finite $p$-groups for distinct primes $p$. Using this fact together with Corollary~\ref{coro:randomizingstability}, it is enough to consider the case where $G$ is a $p$-group (because an abelian CA on ${G_1\times G_2}$ where $G_1$ and $G_2$ have relatively prime orders is a Cartesian product of abelian CAs on $G_1$ and $G_2$ respectively).

  As usual, we write $F$ as
  \[F = \sum_{i\in \Ne} \confendo_i\circ\sigma^i.\]
  Consider the commutative ring generated by the
  ${\confendo_i}$ and the shift map under addition and
  composition. This ring has characteristic $p^l$ for some $l$ because
  we considered a $p$-group as the alphabet. By Lemma~\ref{lem:binomialeries}, we get
  \begin{equation}\label{eq:commuting}\forall n\in\N,\ F^{p^{n+l-1}} = \left(\sum_{i\in
      \Ne}(\confendo_i)^{p^n}\circ\sigma^{ip^n}\right)^{p^{l-1}} = \sum_{j\in \Ne'}(\gamma_j)^{p^n}\circ\sigma^{jp^n},\end{equation}
  where
  \[\Ne' = \{n_{1}+\cdots +n_{p^{l-1}} : n_{i}\in \Ne\}\]
  and each $\gamma_j$ is a sum of compositions of some ${\alfaendo_i}$ that do not depend on $n$. The number of terms in the right-hand expression
  is bounded independently of $n$, so the number of dependencies of
  $F^t$ is bounded on an infinite sequence of times and so $F$ cannot
  be strongly randomizing by Proposition~\ref{prop:rankdeps} (it
  cannot even strongly randomize cylinders of size 1).\bigskip

  For the second part of the proposition, consider ${N=p^{n_0+l-1}}$ for
  some $n_0$ such that ${p^{n_0}>|G|}$. This choice of $N$ guaranties
  that, for any endomorphism $h$ of $G$, we have
  ${\ker(h^{p^{n_0}}) = \ker(h^{p^{n}})}$ for any ${n\geq n_0}$
  (because the sequence ${\ker(h^i)}$ increases strictly until it
  stabilizes). We have the following alternatives:

  \begin{enumerate}[(a)]
  \item \emph{There is $c$ of rank $1$ such that ${\rank\bigl(F^{N}(c)\bigr) = 0}$}. In particular, $F$ is not surjective. We claim that there is some ${g\in G}$ such that ${F^{-t}([g])}$ is empty for any large enough $t$. From the claim we deduce that $F$ doesn't randomize cylinders of length $1$ starting from the uniform Bernoulli measure. We now prove the claim: suppose that ${F^{-t}([g])}$ is never empty, whatever $g$ and $t$. Then consider any finite word ${g_1\cdots g_k}$ and take $n$ large enough so that ${p^n>k}$. By Equation~\ref{eq:commuting} and by choice of $n$ we get that for any configuration $c$, ${F^{p^{n+l-1}}(c)_i}$ depends only on ${c_{|i+V_n}}$ where ${V_n\subseteq \Z}$ is such that ${1+V_n}$, ${2+V_n}$,..., ${k+V_n}$ are two-by-two disjoint sets. Therefore from the assumption that ${F^{-p^{n+l-1}}([g_i])\neq\emptyset}$ for any ${1\leq i\leq k}$ and by linearity of $F$ we deduce that there is some $c$ with ${F^{p^{n+l-1}}(c)\in[g_1\cdots g_k]}$. Since the choice of ${g_1\cdots g_k}$ was arbitrary, we proved that $F$ is surjective which is a contradiction.
% Then, for any ${t\geq N}$ and any ${n\geq 1}$, ${\bigl|\ker(F^{t})\bigr|}$ is at least ${2^n}$ when considering $F$ restricted to periodic configurations of period $n$. In other words, for any ${t\geq N}$ and by taking $n=t\cdot D$ where $D$ is the diameter of $\Ne$, $F^{-t}([0])$ is the (disjoint) union of at least $2^n$ cylinders corresponding to words of length $n$. We deduce that ${F^t\lambda([0]) \geq 2^n\cdot \frac{1}{|G|^n} \sim \frac{2}{|G|}}$ so that $F$ doesn't randomize cylinders of length $1$ starting from the uniform Bernoulli measure.
  \item \emph{${\rank\bigl(F^{N}(c)\bigr) > 0}$ for any $c$ of rank $1$, but there is $d$ of rank $1$ such that ${\rank\bigl(F^{N}(d)\bigr) = 1}$.} By Equation~\ref{eq:commuting}, there is some ${g\in G}$ and a ${j\in \Ne'}$ such that ${g\not\in\ker(\gamma_j^{p^{n_0}})}$ but ${g\in \ker(\gamma_{j'}^{p^{n_0}})}$ for all ${j'\neq j}$. As said before, the choice of $n_0$ ensures that for any ${n\geq n_0}$ we have ${g\not\in\ker(\gamma_j^{p^{n}})}$ but ${g\in \ker(\gamma_{j'}^{p^{n}})}$ for all ${j'\neq j}$. Hence, using the formula for ${F^{p^{n+l-1}}}$, we deduce ${\rank\bigl(F^{p^{n+l-1}}(d)\bigr) = 1}$ for any ${n\geq n_0}$. In that case $F$ admits a soliton of size $1$.
  \item \emph{For any $c$ of rank $1$, ${\rank\bigl(F^{N}(c)\bigr) \geq 2}$.} For the same reason as in the previous case we deduce ${\rank\bigl(F^{p^{n+l-1}}(c)\bigr) \geq 2}$ for any ${n\geq n_0}$. But Equation~\ref{eq:commuting} above shows that the nonzero cells in ${F^{p^{n+l-1}}(c')}$ belong to the set ${\Ne'_n = \{jp^n : j\in \Ne'\}}$ for any $c'$ of rank $1$. We deduce that for any $d$ of rank $m$ and $n$ large enough, ${F^{p^{n+l-1}}(d)}$ contains $2$ nonzero cells distant from each other by at least ${p^n-m}$ cells. Therefore $F$ does not have any soliton and it randomizes harmonically mixing measures in density by Theorem~\ref{thm:main}.
  \end{enumerate}

  To sum up, we have $(c) \Rightarrow (i)$ while $(a)$ and $(b)$ are both incompatible with $(i)$. Since $(c)$ corresponds to $(ii)$, we have shown that $(i)\Leftrightarrow (ii) (\Leftrightarrow (c))$. Since clearly $(i)\Rightarrow (iii)$, we now prove that $(iii)\Rightarrow (c)$.

  If $F$ randomizes in density on cylinders of length $1$, then $F$ is character-diffusive on characters or rank $1$ (by Proposition~\ref{prop:rankdeps}) which means that $\dual{F}$ has no soliton of size $1$. It is straightforward to check by Equation~\ref{eq:defdual} that an abelian CA with commuting endomorphisms has a dual with commuting endomorphisms, and therefore the above alternative (a/b/c) applies also to $\dual{F}$. In other words, $\dual{F}$ satisfies $(c)$, which implies that it admits no solitons. This means in turn that $F$ must satisfy $(c)$ as well. The theorem follows.
\end{proof}

\begin{remark}

  In \cite{PivatoYassawi2} the authors consider abelian CA with integer coefficients (\textit{i.e.} endomorphisms of the form ${a\mapsto n\cdot a}$); such a CA is called \emph{proper} if for any prime divisor $p$ of the order of the alphabet there are at least $2$ coefficients not divisible by $p$. Theorem~6 of \cite{PivatoYassawi2} shows that proper CA are randomizing in density. This is a particular case of the above theorem. Indeed, if $F$ is proper and taking $N$ from the theorem, it is easy to check that $F^N$ is proper and that this is equivalent to $\rank(F^N(c))\geq 2$ for any finite configuration of rank $1$.

%  Indeed, if $F$ is proper then so is $F^2$ and by immediate induction any $F^n$ for ${n\in\N}$. Consider now some finite configuration $c$ of rank $1$ and let ${a=c(z)\neq 0}$ be its unique nonzero state. If $p$ divides the order of $a$ then each coefficient of $F^n$ which is not congruent to $0$ modulo $p$ produces a distinct nonzero state in ${F^n(c)}$. We deduce that ${\rank{F^N(c)}\geq 2}$ for any proper CA $F$ and any $c$ of rank $1$.
\end{remark}

\subsection{Strong Randomization}
\label{sec:strongrandom}

We now give examples of strongly randomizing abelian
cellular automata. They are all defined over the alphabet
$\A = {\Z_p^2}$ where $p$ is a prime number. In this
section we denote by $\pi_1$ and  $\pi_2$ the projections on the first
and second component of the alphabet respectively.

%To simplify we use the same notation for elements of the alphabet and for configurations.

 We also denote ${n\cdot g =
  \underbrace{g+\cdots+ g}_{n\text{ times}}}$ and
$\overline{0}$ the neutral element of the group
${\bigl(\A^\Z,+\bigr)}$ (\textit{i.e.} the configuration
everywhere equal to $(0,0)$).

\begin{lemma}
  \label{lem:binomialmodp}
  If two abelian CA $F$ and $G$ over $\A^\Z$ commute, then for any $n\geq 0$
  we have:
  \[(F+ G)^{p^n}=F^{p^n}+ G^{p^n}\]
\end{lemma}
\begin{proof}
  Due to the structure of the group $\A=\Z_p^2$, for any configuration $c\in \A^\Z$ we have ${p\cdot c = \overline{0}}$. Therefore for any CA $F$, we have that ${p\cdot F}$ is the constant map equal to $\overline{0}$.
  
  Now, from the binomial formula and from the fact that $p$ divides
  ${p \choose n}$ for any ${1<n<p}$, we deduce that:
  \[(F+ G)^p = F^p + G^p.\]
  The lemma follows by an easy induction.
\end{proof}

Extending example $F_2$ from Section~\ref{sec:defs}, we now consider for each prime $p\geq 2$ and all ${c\in \A^\Z}$:
\begin{align*}
  F_p(c)_z &= \bigl(\pi_1(c_{z+1})+\pi_2(c_z),
  \pi_1(c_z)\bigr)\\
  G_p(c)_z &= \bigl(\pi_1(c_{z+1})+\pi_1(c_{z})+\pi_2(c_z),
  \pi_1(c_z)\bigr)
  % H_p(c)_z &= \bigl(\pi_1(c_{z-1})+\pi_1(c_{z+1})+\pi_2(c_z),
  % \pi_1(c_z)\bigr)
\end{align*}

In the remainder of the section we prove that $F_p$ and $G_p$ are strongly diffusive for any prime ${p\geq 2}$.\bigskip

These cellular automata are reversible and in fact also
time-symmetric \textit{i.e.} the product of two involutions \cite{GajardoKM12}. For instance the inverse of $F_p$ is:
\[F_p^{-1}(c)_z = \bigl(\pi_2(c_z), \pi_1(c_z)-\pi_2(c_{z+1})\bigr).\]

Since they are reversible, we extend their dependency diagrams to negative times, which means that $\Delta_\Phi(t,z)$ is defined for any ${(t,z)\in\Z^2}$ where $\Phi$ denotes $F_p$ or $G_p$.

Note that both $F_p$ and its inverse can be defined with neighborhood
${\{0,1\}}$. The same is true for $G_p$.

\begin{lemma}
  \label{lem:selfsimilarity}
  For any $n\geq 0$, any $t\in\Z$ and any $z\in\Z$ we have:
  \begin{align*}
    \Delta_{F_p}(2p^n+t,z) &= \Delta_{F_p}(p^n+t,p^n+z)+\Delta_{F_p}(t,z)\\
    \Delta_{G_p}(2p^n+t,z) &= \Delta_{G_p}(p^n+t,p^n+z)+ \Delta_{G_p}(p^n+t,z)+\Delta_{G_p}(t,z)\\
  \end{align*}
\end{lemma}
\begin{proof}
  First, it is straightforward to check that ${F_p^2 = (\sigma\circ F_p)
    + I}$ (where $I$ denotes the identity map over
  $\A^\Z$). Hence, using Lemma~\ref{lem:binomialmodp}, we get the
  identity ${F_p^{2p^n} = (\sigma_{p^n}\circ F_p^{p^n})+
    I}$. For every configuration $c$, $t\in\Z$ and $z\in\Z$ we have
${F_p^{2p^n+t}(c)_z = F_p^{p^n+t}(c)_{p^n+z}+  c_z}$, which proves the Lemma.

  The same proof scheme applies to $\Delta_{G_p}$.
\end{proof}

\begin{lemma}
  \label{lem:basicdeps}
  Let $\Phi$ be either $F_p$ or $G_p$. For any ${t\in\Z}$ we have:
  \begin{enumerate}
  \item $\Delta_\Phi(t,z)$ is the constant map equal to $(0,0)$ when
    ${z>0}$ or ${z<-|t|}$
  \item $\Delta_\Phi(t,0)$ is a bijection.
  \end{enumerate}
\end{lemma}
\begin{proof}
  First both $\Phi$ and $\Phi^{-1}$ have neighborhood ${\{0,1\}}$. So
  the first item is straightforward by induction. 

  Second, by definition, $\Delta_\Phi(0,0)$ is a bijection. We can check that:
  \[\Delta_{F_p}(1,0):(g,h)\mapsto (h,g),\ \Delta_{G_p}(1,0):(g,h)\mapsto (g+h,g),\ \Delta_{G_p}(2,0):(g,h)\mapsto (h,g+h),\]
  which are bijections. Applying Lemma~\ref{lem:selfsimilarity} with $n=0$, we get by straightforward induction that $\Delta_{F_p}(t+2,0)=\Delta_{F_p}(t,0)$ and $\Delta_{G_p}(t+3,0) = \Delta_{G_p}(t,0)$. We proved the second item.
\end{proof}

Much of the structure of $\Delta_\Phi$ can be understood when focusing
on particular ``triangular'' zones of $\Z^2$ at various scales. For $k\geq
0$ and $n$ large enough so that ${p^n-k>k}$, we define the corresponding zone as:
\[T_{n,k} = \{(t,z) : k<t<p^n-k\text{ and }-t<z<k\}\]

\begin{lemma}
  \label{lem:rankplusone}
  Let $\Phi$ be either $F_p$ or $G_p$.
  Let ${k\geq 0}$ and $n$ such that ${p^n-k>k}$. Then for any
  $(t,z)\in T_{n,k}$ and any ${j\geq 1}$ we have:
  \[\Delta_\Phi(t,z) = \Delta_\Phi(t+j\cdot p^n,z-j\cdot p^n).\]
  In particular, if $\chi$ is a character whose support is of diameter at
  most $k$, then for any $t$ with ${k<t<p^n-k}$ we have:
  \[\rank(\chi\circ\Phi^{t+j\cdot p^n})\geq \rank(\chi\circ\Phi^t)+1\]
\end{lemma}
The lemma is illustrated in Figure~\ref{fig:triangles}.

\begin{figure}
\begin{center}
  \reflectbox{\includegraphics[width=0.8\textwidth, angle=180]{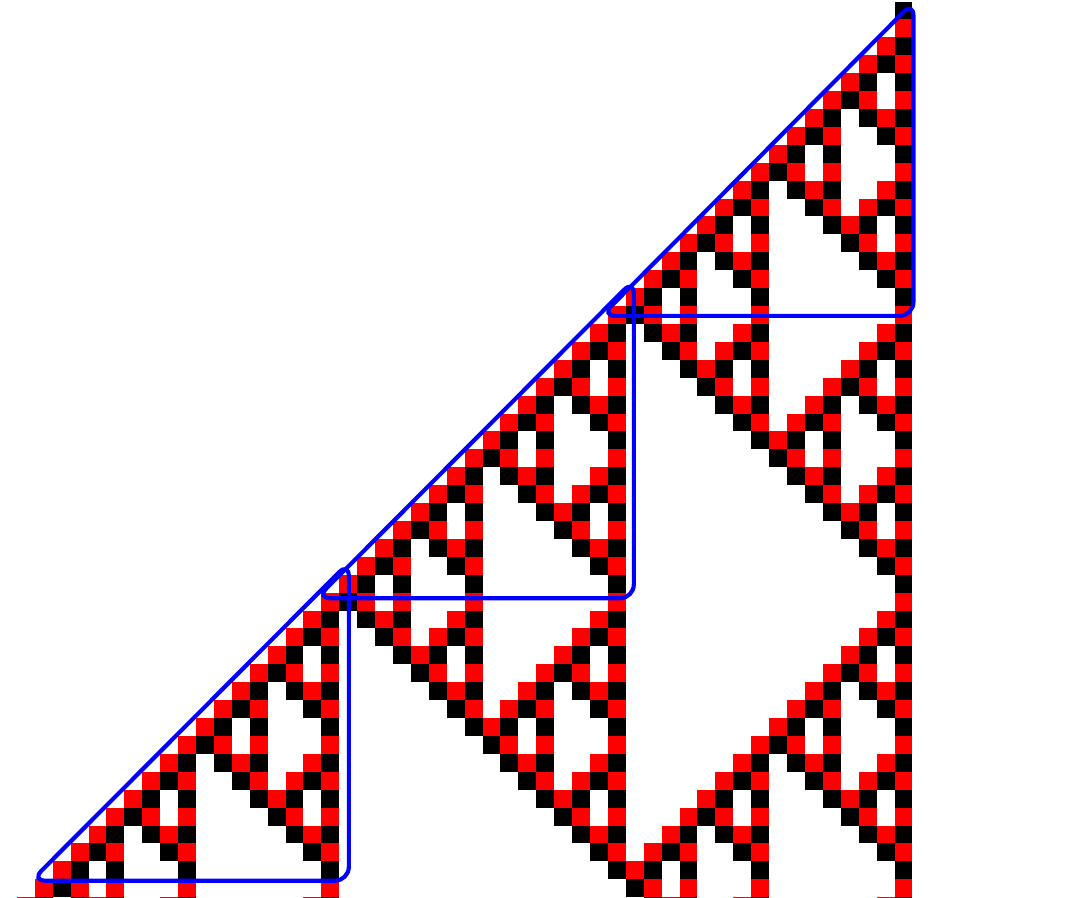}}
  \caption{A representation of $\Delta_{F_2}(i,t)$ for $i\in\Z$ and $t\in\N$. The direction of time is upward. Thick blue lines indicate triangular zones $T_{4,0}$ and its translations by $(j\cdot 2^4, -j\cdot 2^4)$ for $j=1,2$.}\label{fig:triangles}
\end{center}
\end{figure}

\begin{proof}
  For the first assertion consider some $(t,z)\in T_{n,k}$. From
  Lemma~\ref{lem:basicdeps}, we have that both
  ${\Delta_\Phi(t-p^n,z-p^n)}$ and ${\Delta_\Phi(t,z-p^n)}$ are
  the constant map equal to $(0,0)$. Therefore the following identities obtained from Lemma~\ref{lem:selfsimilarity}
  \begin{align*}
    \Delta_{F_p}(t+p^n,z-p^n) &= \Delta_{F_p}(t-p^n,z-p^n) + \Delta_{F_p}(t,z)\\
    \Delta_{G_p}(t+p^n,z-p^n) &= \Delta_{G_p}(t-p^n,z-p^n) + \Delta_{G_p}(t,z-p^n)
    + \Delta_{G_p}(t,z)
  \end{align*}
  can in both cases be simplified to:
  \[\Delta_\Phi(t,z) = \Delta_\Phi(t+p^n,z-p^n).\]  
  This prove the case ${j=1}$. The same idea shows the
  induction step on $j$:
  \[\Delta_\Phi(t+j\cdot p^n,z-j\cdot p^n) = \Delta_\Phi(t+(j+1)\cdot p^n,z-(j+1)\cdot p^n).\]

  For the second assertion, apply Lemma~\ref{lem:localrank} on the first assertion. We get that $-z\in\support(\chi\circ\Phi^t)$ if and only if $-z+j\cdot p_n \in \support(\chi\circ\Phi^{t+j\cdot p^n})$. Furthermore, we also have $0 \in \support(\chi\circ\Phi^{t+j\cdot p^n})$ since $0$ is a $k$-isolated bijective dependency by Lemma~\ref{lem:basicdeps}. Accounting for both contributions, we get $\rank(\chi\circ\Phi^{t+j\cdot p^n})\geq \rank(\chi\circ\Phi^t)+1.$ \end{proof}

%\TODO{give some intuition...}

\begin{theorem}\label{ref:strongrandomize}
  $F_p$ and $G_p$ strongly randomize the harmonically mixing measures.
  \end{theorem}

  \begin{proof}
We prove that $F_p$ and $G_p$ are strongly character-diffusive, and Proposition~\ref{prop:chidiffusive} implies the result.
    
  Denote by $\Phi$ either $F_p$ or $G_p$. $\Phi$ is reversible and denote $C$ the diameter of the neighborhood of $\Phi^{-1}$. By Lemma~\ref{lem:rankdecreasing}, $\rank(\chi\circ \Phi)\geq C\cdot\rank(\chi)$ for any character $\chi$. Since $\Phi^{-t}$ can be defined by a neighborhood of diameter $C\cdot t$, $\rank(\chi\circ \Phi^t)\geq C\cdot t\cdot\rank(\chi)$ for any character $\chi$. In particular, for any $m\geq0$, any $T\geq 0$, and any character $\chi_0$ of rank at least $C\cdot m\cdot T$, we have $\rank{(\chi_0\circ\Phi^t)}\geq m$ for ${1\leq t\leq T}$.
    
    Let $\chi$ be any non-trivial character and let $k$ be the diameter of its support. Denote by ${R(t)=\rank(\chi\circ\Phi^t)}$. We are going to show that ${R(t)\rightarrow\infty}$, which implies the claim since the choice of $\chi$ is arbitrary.

  First, let $n_0$ be large enough and $t_0$ such that
  ${k<t_0<p^{n_0}-k}$. By successive applications of
  Lemma~\ref{lem:rankplusone}, we get
  \[R\bigl(t_0+(p-1)p^{n_0}+(p-1)p^{n_0+1}+\cdots
  +(p-1)p^{n_0+m}\bigr)\geq m.\]

  Since ${p^{n_0+m+1}=\sum_{j=0}^m(p-1)p^{n_0+j}+p^{n_0}}$ we deduce that 
  \begin{equation*}
    p^{n_0+m+1} - \bigl(t_0 + \sum_{j=0}^m(p-1)p^{n_0+j}\bigr) \leq p^{n_0}-t_0
  \end{equation*}
  so that as soon as $m\geq C\cdot M\cdot (p^{n_0}-t_0)$ it follows $R(p^{n_0+m+1}) \geq M$ by Lemma~\ref{lem:rankdecreasing}. Therefore:
    \begin{equation}
    R(p^{n})\rightarrow_n\infty.
    \label{eq:growatpowers}
  \end{equation}
  Now define the predicate $P_{t_0,n,m}$ as the conjunction of the following conditions:
  \begin{enumerate}
  \item ${k\leq t_0\leq p^n - k}$;
  \item $\forall t, t_0\leq t\leq p^n-k: R(t)\geq m$;
  \item $R(p^n-k)\geq C\cdot(m+1)\cdot(t_0+k)$.
  \end{enumerate}
  First, since $R(t)\geq 1$ for any $t$ (from
  Lemma~\ref{lem:basicdeps}), we have by Equation (\ref{eq:growatpowers}) above
  and Lemma~\ref{lem:rankincrease} that $P_{k,n,1}$ holds for $n$ large enough.
  
  Furthermore, if $P_{t_0,n,m}$, then $P_{t_0+p^n,n',m+1}$ for all
  large enough $n'$. Indeed, condition 1 is obviously true for $n'>
  n$, condition 3 is true for any large enough $n'$ from
  (\ref{eq:growatpowers}) above and
  Lemma~\ref{lem:rankincrease}. Finally, condition 2 is obtained from
  Lemma~\ref{lem:rankplusone}: for any $j\geq 1$,
  \begin{enumerate}
  \item for $j\cdot p^n +t_0\leq t\leq (j+1)\cdot p^n - k$, ${R(t)\geq R(t-j\cdot p^n) + 1\geq m+1}$;
  \item for $(j+1)\cdot p^n - k\leq t\leq (j+1)\cdot p^n+t_0$ we have ${R(t)\geq m+1}$ because ${R((j+1)\cdot p^n-k)\geq R(p^n-k)\geq C\cdot(m+1)\cdot(t_0+k)}$.
  \end{enumerate}
  
  We have shown that for any $m$ there exists $t_0$ such that
  for any large enough $n'$ we have $P_{t_0,n',m}$. This in particular
  implies ${R(t)\geq m}$ for all $t\geq t_0$. We conclude that
  ${R(t)\rightarrow\infty}$.\end{proof}

\subsection{Randomizing only up to fixed-length cylinders}
\label{sec:fixedlength}

We now define a family of cellular automata that randomize finite cylinders up to a certain length, but no further. The alphabet is $G^2$ where $G$ is any finite abelian group. Define:
\begin{align*}
  I_G(c)_z &= \bigl(-[\pi_1(c_{z-1})+\pi_1(c_{z+1})+\pi_2(c_z)],
  \pi_1(c_z)\bigr).
\end{align*}
Notice that $I_{\Z_2} = H_2$.

\begin{proposition}\label{prop:cellwise}
$I_G$ randomizes cylinders of length $1$, but not cylinders of length $2$.
\end{proposition}

\begin{lemma}\label{lem:IG}
For $t>0$, we have: 
\[\Delta_{I_G}(t,z) = (g,h)\mapsto \left\{
\begin{array}{ccc}
0 & \mbox{if} &|z|>t\\
({(-1)^t}g,0)&\mbox{if}& |z|=t\\
({(-1)^t}g, {(-1)^{t+1}}h) & \mbox{if}&|z|<t, t+z = 0 \mod 2\\
({(-1)^t}h, {(-1)^{t+1}}g) & \mbox{if}&|z|<t, t+z = 1 \mod 2
\end{array}\right.\]
\end{lemma}
\begin{proof}
By straightforward induction.
\end{proof}

Now we use Proposition~\ref{prop:chidiffusive} in conjunction with the following proposition to prove the announced result. 
\begin{proposition}\label{prop:IGdiffusive}
$I_G$ is strongly diffusive on characters of rank 1, but not on characters of length 2.
\end{proposition}

\begin{proof}
Let $\chi$ be a character of rank 1, i.e. $\chi(x) = \chi_0(x_0)$ with $\chi_0\neq \textbf{1}$. For $t>0$:
\begin{align*}\chi\circ I_G^t(x) &= \chi_0\left(\sum_{z=-t}^t\Delta_{I_G}(t,z)(x_{-z})\right)\\
&=\prod_{z=-t}^t \chi_0\circ\Delta_{I_G}(t,z)(x_{-z})
\end{align*} 
$\chi_0$ is nontrivial and by the previous Lemma, $\Delta_{I_G}(t,z)$ is an isomorphism when $|z|<t$. We deduce that $\chi_0\circ\Delta_{I_G}(t,z)$ is nontrivial whenever $|z|<t$ and therefore $\rank(\chi\circ I_G^t(x))\geq 2t-1 \to \infty$: $I_G$ strongly diffuses characters of rank 1.

%For the second point, define $x$ by $x_0 = (0,g),\ x_1=(g^{-1},0)$ for some arbitrary $g\neq 0$, and $x_i=0$ otherwise. Then check by hand that $F(x) = (\s(x))^{-1}$ (group inverse) and $F^2(x)=\s^2(x)$.

For the second point, take any elementary character $\eta_0 \in \dual{G}$ and define another character $\eta : x\mapsto \eta_0(\pi_1(x_0)+ \pi_2(x_1))$. Then, by a straightforward computation:

\begin{align*}\eta\circ I_G(x) =& \eta_0\bigl(-[\pi_1(x_{-1})+\pi_1(x_{1})+\pi_2(x_0)]+\pi_1(x_1)\bigr)\\
=&\eta_0\bigl(-\pi_1(x_{-1}) - \pi_2(x_0)\bigr) = \s^{-1}\circ \eta(x)^{-1}
\end{align*}
which means $\eta$ a soliton for $\dual F$ of rank $2$.
\end{proof}

\newcommand\blowup{I_{G,n}}%\mathstrut^{(n)}I_G}
Now we introduce the cellular automaton $\blowup$, which consists in applying the local rule of $I_G$ on the neighbourhood $\{z_{-n}, z_0, z_n\}$:

%\TODO{This notation is awful so I defined a macro. Suggestions?}
\begin{align*}
  \blowup(c)_z &= \bigl(-[\pi_1(c_{z-n})+\pi_1(c_{z+n})+\pi_2(c_z)],
  \pi_1(c_z)\bigr)
\end{align*}
Intuitively, a space-time diagram for $\blowup$ consists of $n$ intertwined space-time diagrams for $I_G$.

\begin{theorem}
For any $n \geq 1$, $\blowup$ randomizes cylinders of length $n$, but does not randomize cylinders of length $n+1$.
\end{theorem}

It is obvious (straightforward induction) that: 
\begin{lemma}\label{lem:blowup}
\[\Delta_{\blowup}(t,z) = \left\{\begin{array}{ll}\Delta_{I_G}(t,z/n)&\mbox{if }z=0\mod n\\0&\mbox{otherwise}\end{array}\right.\]
\end{lemma}
so that we can use Lemma~\ref{lem:IG}. As in the previous case, we use Proposition~\ref{prop:chidiffusive} in conjunction with the following:
\begin{proposition}
$\blowup$ is strongly \chidiffusive on characters of support $\subset [0,n-1]$, but not on characters of support $\{0,n\}$.
\end{proposition}

\begin{proof}
Let $\chi$ be a nonzero character of support $\subset [0,n-1]$, that is, $\chi = \prod_{i=0}^{n-1}\chi_i$; without loss of generality assume $\chi_0\neq \textbf{1}$. We have for any $t>0$:
\begin{align*}
\chi\circ \blowup^t(x) &= \prod_{i=0}^{n-1}\chi_i\left(\sum_{z=-t}^t\Delta_{\blowup}(t,z+i)(x_{-z})\right)\\
&=\prod_{z=-t}^t\left(\prod_{i=0}^{n-1}\chi_i\circ \Delta_{\blowup}(t,z+i)\right)(x_{-z})
\end{align*} 
Now for any $z$ such that $z = 0 \mod n$, the corresponding term in the previous equation is $\chi_0\circ \Delta_{\blowup}(t,z)$ by Lemma~\ref{lem:blowup}, and this term is an isomorphism when $t<z<t$ by Lemma~\ref{lem:IG}. Therefore $\rank(\chi\circ \blowup^t)\geq 2\left\lfloor\frac tn\right\rfloor-1 \to \infty$.

%For the second point, define $x$ by $x_0 = (0,g),\ x_n=(g^{-1},0)$ for some arbitrary $g\neq 0$, and $x_i=0$ otherwise, and check that $F^2(x) = \s^{2n}(x)$.
For the second point, take any nontrivial elementary character $\eta_0 \in \dual{G}$ and define another character $\eta : x\mapsto \eta_0(\pi_1(x_0)+ \pi_2(x_n))$. This is a soliton for $\dual F$ of support $\{0,n\}$ by Lemma~\ref{lem:blowup} and the same proof as Proposition~\ref{prop:IGdiffusive}.
\end{proof}

\section{Open problems}
\label{sec:open}

Building upon the approach of \cite{PivatoYassawi1,PivatoYassawi2} we completely characterized randomization in density for abelian CAs. Furthermore, we provided examples of other forms of randomization, most notably strong randomization (in simple convergence), that can only happen for abelian CA whose coefficient are noncommuting endomorphisms.

As mentioned by several authors, the most important research direction for randomization in CA is to develop tools and techniques to go beyond the abelian case: CA with a nonabelian group structure or nonlinear CA~\cite{host2003uniform,Kari2015}. There is some experimental evidence pointing at nonlinear randomization candidates \cite{These, Taati}. The class of bipermutive CA, although it does not encompass all candidates, has some relevant related work: the set of invariant well-behaved (i.e. Gibbs) measures is limited to the uniform Bernoulli measure (\cite{Kari2015}, Corollary 46) and it exhibits a topological analogue to randomization \cite{SaloTorma}. Apart from the existence of examples, one can ask in the general case whether randomization is a structural property (see Proposition~\ref{prop:chidiffusive}), whether it is preserved by elementary operations and whether it is implied by some topological properties like positive expansivity (see Corollary~\ref{coro:randomizingstability}).

We believe that several intermediate questions raised by the present work are worth being considered:

\begin{itemize}
\item Is strong randomization of an abelian CA equivalent to strong randomization of its dual?
\item What is the importance of reversibility in the above examples of strong randomization? Our proof relies on reversibility and the smoothness it implies on the evolution of rank of characters. Can a reversible abelian CA be randomizing in density but not strongly randomizing? What are the examples of strongly randomizing non-reversible abelian CA?
\item The notions of soliton and diffusivity can be defined for arbitrary CAs using diamonds (like in the notion of pre-injectivity): a pair ${(x,y)}$ such that ${\Delta(x,y)\subseteq\Z}$ (the set of positions where $x$ and $y$ differ) is finite. We can define a soliton as a diamond $(x,y)$ such that ${\Delta(F^t(x),F^t(y))}$ has a bounded diameter (independent of $t$). On the contrary diffusivity can be defined as the property of having ${|\Delta(F^t(x),F^t(y))|\rightarrow\infty}$ for any diamond ${(x,y)}$. What are the links with topological properties like positive expansivity, pre-expansivity\footnote{Positive pre-expansivity, introduced in \cite{pre-exp}, is the property of being positively expansive on diamonds. A reversible CA can be positively pre-expansive (but never positively expansive) like the example $F_2$ of this paper. Corollary~\ref{coro:randomizingstability} states that, for abelian CA, a direction of positive expansivity implies randomization in density: this implication actually holds for directions of positive pre-expansivity.}, mixing or transitivity? Are there links with randomization beyond abelian CAs? The results of \cite{sparseville} can be useful for that line of research.
\item Theorem~\ref{thm:commutingcase} gives a procedure to test randomization in density for abelian CA with commuting coefficients because the constant $N$ can be explicitely computed. Is there an algorithm to decide the presence of a soliton in abelian CA? What about solitons in general CAs (again formalizing through diamonds as above)?
\end{itemize}

\section{Acknowledgments}
We are grateful to the anonymous referee for her/his help to improve this paper and correct several mistakes of the previous version.
Benjamin Hellouin de Menibus acknowledges the financial support of Basal project No. PFB-03 CMM, Universidad de Chile.

\bibliographystyle{plain}
\bibliography{paper_ca}
\end{document}